\documentclass[reqno]{amsproc}
\usepackage[latin1]{inputenc}
\usepackage{amssymb}
\usepackage{hyperref}
\usepackage{amsmath}
\usepackage{latexsym}
\usepackage{cite}
\usepackage{amssymb}
\usepackage{amsfonts}
\usepackage{amscd}
\usepackage{xcolor}
\usepackage{amstext,amsmath,amssymb,amsfonts}
\newtheorem{theorem}{Theorem}

\newtheorem{corollary}[theorem]{Corollary}

\newtheorem{definition}[theorem]{Definition}
\newtheorem{example}[theorem]{Example}

\newtheorem{proposition}[theorem]{Proposition}
\newtheorem{remark}[theorem]{Remark}

\newcommand{\C}{\mathbb{C}}

\newcommand{\K}{\mathbb {K}}

\newcommand{\A}{\mathcal{A}}

\newcommand{\G}{\mathcal{G}}
\newcommand{\B}{\mathcal{B}}
\newcommand{\bb}{\mathfrak{B}}

\newcommand{\h}{\mathcal{H}}

\newcommand{\beq}{\begin{eqnarray}}
\newcommand{\eeq}{\end{eqnarray}}
\newcommand{\beqs}{\begin{eqnarray*}}
\newcommand{\eeqs}{\end{eqnarray*}}
\newcommand{\bpro}{\begin{pro}}
\newcommand{\epro}{\end{pro}}
\newcommand{\blem}{\begin{lem}}
\newcommand{\elem}{\end{lem}}
\newcommand{\bdfn}{\begin{dfn}}
\newcommand{\edfn}{\end{dfn}}
\newcommand{\bcor}{\begin{cor}}
\newcommand{\ecor}{\end{cor}}
\newcommand{\bthm}{\begin{thm}}
\newcommand{\ethm}{\end{thm}}
\newcommand{\bex}{\begin{ex}}
\newcommand{\eex}{\end{ex}}
\newcommand{\brmk}{\begin{rmk}}
\newcommand{\ermk}{\end{rmk}}
\newcommand{\bpr}{\begin{pr}}
\newcommand{\epr}{\end{pr}}
\newcommand{\benum}{\begin{enumerate}}
\newcommand{\eenum}{\end{enumerate}}
\newcommand{\bitem}{\begin{itemize}}
\newcommand{\eitem}{\end{itemize}}
\newcommand{\br}[1]{   [ \cdot,    \cdot  ]   }

\newcommand{\g}{\frak{g}}
\newcommand{\ad}{\mathrm{ad}}

\newcommand{\cqfd}{\hfill{\square}}
\chardef\bslash=`\\
\numberwithin{equation}{section}
\numberwithin{table}{section}
\numberwithin{theorem}{section}

\title[A pre-Jacobi-Jordan algebra]{On a pre-Jacobi-Jordan algebra: relevant properties and double construction\footnote{Preprint: ICMPA-UL/2019/12 } }

\author{Cyrille Essossolim Haliya$^\ast$}
\address[$\ast$]
{University of Abomey-Calavi,
International Chair in Mathematical Physics and Applications,
ICMPA-UNESCO Chair, 072 BP 50, Cotonou, Rep. of Benin}
\email{cyzille19@gmail.com}

\author{Gb\^ev\`ewou Damien  Houndedji$^\dagger$}
\address[$\dagger$]{University of Abomey-Calavi,
International Chair in Mathematical Physics and Applications,
ICMPA-UNESCO Chair, 072 BP 50, Cotonou, Rep. of Benin}
\email{ houndedjid@gmail.com}
\begin{document}
\maketitle

\today

\bigskip
\begin{abstract}
	We introduce a pre-Jacobi-Jordan algebras and study some relevant properties such as  bimodules, matched pairs. Besides, we established a pre-Jacobi-Jordan algebra built as a direct sum of a given pre-Jacobi-Jordan algebra $(\A, \cdot)$ and its dual $(\A^{\ast}, \circ),$ endowed with a 
non-degenerate symmetric bilinear form $B,$ where $\cdot$ and $\circ$ are the products defined on $\A$ and $\A^{\ast},$ respectively. Finally, after pre-Jacobi-Jordan algebras classification in dimension two, we thoroughly give some double constructions of pre-Jacobi-Jordan algebraic structures.
\end{abstract}

\section{Introduction}
 Jacobi-Jordan algebras (JJ algebras for short) were introduced in \cite{[DB]} in 2014 as vector
spaces $\A$ over a field k, equipped with a bilinear map $\cdot: \A\times\A \to \A$ satisfying the Jacobi identity and instead of the skew-symmetry condition valid for Lie algebras the
commutativity $x \cdot y = y \cdot x$, for all $x$, $y \in A$ is imposed. Apparently this class of algebras appear under different names in the litterature reflecting, perhaps, the fact that it was considered from different viewpoints by different communities, sometimes not aware of each other's results. In \cite{[1],[2],[3]} and other Jordan litterature,these algebras are called \textit{Jordan algebras of nil index 3}. In \cite{[OK]} they are called \textit{Lie-Jordan algebras} (superalgebras are also considered there). In \cite{[Pasha]} and \cite{[Zumano]} they were called \textit{mock-Lie algebras}. 

In \cite{[WO]} W\"orz-Busekros relates these type of algebras  with Bernstein algebras. One crucial remark is that JJ algebras
are examples of the more popular and well-referenced Jordan algebras  \cite{[ee],[13]} introduced in
order to achieve an axiomatization for the algebra of observables in quantum mechanics.
In \cite{[DB]} the authors achieved the classification of these algebras up to dimension 6 over
an algebraically closed field of characteristic different from 2 and 3.  As it was explained in the paper and  in  \cite{[ee]}, JJ algebras are objects fundamentally different from both associative and
Lie algebras though their definition differs from the latter only modulo a sign. In  \cite{[ee]} it's proven that there exists a rich and very interesting theory behind the JJ
algebras which deserves to be developed further mainly for three important reasons (see  \cite{[ee]} for more details). They mainly discuss the general extension (GE) problem (which is a kind of generalization of the classical Holder extension problem) for JJ algebras. Interestingly, they authors prove that a finite dimensional JJ algebras is Frobenius if and only if there exists an invariant non degenerate bilinear form (Proposition 1.8). 

 A  Frobenius algebra  is an associative
algebra equipped with a  non-degenerate invariant bilinear form. This type of algebras plays 
an important role in different areas of   mathematics and 
physics,  such as
statistical models over two-dimensional graphs \cite{[MB]} and topological quantum field theory
\cite{[JK]}. On the other hand, an antisymmetric bilinear form on an associative algebra $\mathcal A$ is an antisymmetric bilinear form on $\mathcal A$ which is a $1-$cocycle, or 
  Connes cocycle, for the Hochschid cohomology. 

 In \cite{[C.Bai5]},  C. Bai described associative analogs of Drinfeld's double constructions for Frobenius algebras and for associative algebras equipped with  non-degenerate Connes
cocycles. 
We note that there are two different types of constructions involved:
\begin{enumerate}
\item[i)] the Drinfeld's double type constructions, from a Frobenius algebra or
from an associative algebra equipped with a Connes cocyle, and
\item[ii)] the Frobenius algebra obtained from an anti-symmetric solution of the
associative Yang-Baxter equation and the non-degenerate Connes cocycle
obtained from a symmetric solution of a D-equation.
\end{enumerate}


Using the main fact that JJ algebras are Frobenius under susmentioned condition, our purpose is to proceed a double construction for JJ algebras and for pre-JJ algebras.

		We use in this paper the double construction's technique by Bai in \cite{[C.Bai5]}. The aim of this paper is to give some basics of JJ algebras and, study pre-JJ algebras specially their bimodules and matched pairs. The
double construction of symmetric (pre)JJ algebras.  Furthermore, a two dimensional classification of pre-JJ algebras 
is given with a special emphasis on the corresponding double construction.


  The paper is organized as follow. In Sec.2, after recalling some basics concepts and necessary properties on antiassociative and JJ algebras, we discuss on the   matched pairs of JJ algebras. On the other hand we discuss some basics properties of pre-JJ algebras. Sec.3 is devoted for bimodules and matched pairs of pre-JJ algebras. In Sect.4, we give results on double construction symmetric (pre)JJ algebras. A detailed survey is done for a two dimensional pre-JJ algebras in Sec.5. In Sec.6, we end with some concluding remarks.

\section{A pre-Jacobi-Jordan algebra: definition and main results}

\subsection{Representations and Matched pair of Jacobi-Jordan algebras}
 Throughout this work, we consider  $\A$  a finite dimensional  vector space over the field $ \K $ of 
 characteristic different from $2,3$ together with a bilinear product $"\cdot"$ defined as 
 $\displaystyle 
  : \A \times \A \rightarrow~\A$ such that $\; (~x~,~y~) \mapsto~x~\cdot~y.$
  \begin{definition} \cite{[OK]}
Let "$\cdot$" be a bilinear product in a vector space $\A.$ Suppose that it satisfies the following law:
\beq\label{ant}
(x\cdot y)\cdot z= -x\cdot(y\cdot z).
\eeq
Then, we call the pair $(\A, \cdot)$ an \textbf{antiassociative algebra}. 
\end{definition}
\begin{definition}\cite{[Zumano]}
An algebra $(\A, \diamond)$ over $K$ is called JJ if it is commutative:
\beq
x\diamond y=y\diamond x,
\eeq
and satisfies the Jacobi identity:
\beq
(x\diamond y)\diamond z+(z\diamond x)\diamond y+(y\diamond z)\diamond x=0
\eeq
for any $x$,$y$,$z\in \A$.
\end{definition}

\begin{theorem}\cite{[Zumano]}
Given an antiassociative algebra $(\A,\cdot)$, the new algebra $\A^{\dagger}$ with multiplication give by the "anticommutator"
\beqs
a\diamond b=\frac{1}{2}\left(a\cdot b+b\cdot a\right),
\eeqs
is a JJ algebra.
\end{theorem}
Since JJ algebras are commutative, the left and right actions of an algebra coincide, so we can speak about just modules.
\begin{definition}\cite{[Zumano]}
A vector space $V$ is a module over a JJ algebra $\A$, if there is a linear map (a representation) $\rho:\A\to End(V)$ such that
\beq
\rho(x\diamond y)(v)=-\rho(x)(\rho(y)v)-\rho(y)(\rho(x)v)
\eeq
for any $x,y\in\A$ and $v\in V$.
\end{definition}
\begin{proposition}\label{prop2}
Let $(\g,\diamond)$ be a
JJ algebra and  $(V,\rho)$ be a representation of
$\mathfrak{g}$.

The direct summand $\g\oplus V$ with a bracket defined by
\begin{equation}
(x+u)\diamond(y+w):= x\diamond y + \rho(x)(w)+\rho (y)(u)  \quad \forall
x,y\in \g \ \forall u,w\in V
\end{equation}
is  a JJ algebra.
\end{proposition}

\begin{proof}
The symmetry of the bracket is obvious. We show that the
Jacobi identity is satisfied:

Let $ x,y,z\in \g$ and $\forall u,v,w\in V.$
\begin{eqnarray*}
&&\circlearrowleft_{(x,u),(y,v),(z,w)}(x+u)\diamond((y+v)\diamond(z+w))\cr
&=&\circlearrowleft_{(x,u),(y,v),(z,w)}(x+u)\diamond(y\diamond z +\rho (y)(w)+\rho (z)(v))\\
&=&\circlearrowleft_{(x,u),(y,v),(z,w)} x\diamond(y\diamond z) +\rho (x)(\rho
(y)(w)+\rho (z)(v))+\rho (y\diamond z)(u)\\
&=&\circlearrowleft_{(x,u),(y,v),(z,w)}\rho (x(\rho
(y)(w))+\rho (x(\rho (z)(v))+\rho (y(\rho
(z)(u))\\
&+&\rho (z(\rho (y)(u))\\
  &=& \rho (x(\rho
(y)(w))+\rho (x(\rho (z)(v))+\rho (y(\rho
(z)(u))\rho (z(\rho (y)(u))\\
&+&\rho (y(\rho (z)(u))+\rho (y(\rho (x)(w))+\rho
(z(\rho (x)(v))+\rho (x(\rho (z)(v))
\\
&+&\rho (z(\rho (x)(v))+\rho (z(\rho (y)(u))+\rho
(x(\rho (y)(w))+\rho (y(\rho (x)(w))\\
&=&0
\end{eqnarray*}
where $\circlearrowleft_{(x,u),(y,v),(z,w)}$ denotes summation over
the cyclic permutation on \\
$(x,u),(y,v),(z,w)$.
\end{proof}
\begin{definition}
Let $(\g,\diamond)$ be a JJ algebra. Two representations $(V,\rho)$ and $(V',\rho')$  of $\mathfrak{g}$ are said to be isomorphic if there exists a linear map
$\phi\ :V \rightarrow V^{\prime }$ such that
$$
 \forall x\in \g, \ \ \rho ^{\prime }(x)\circ \phi=\phi\circ \rho(x).
$$
\end{definition}
\begin{proposition}
Let $(\g, \diamond)$ be a JJ algebra
and $L_{\diamond}: \g\rightarrow End(\g)$ be an operator defined for $x\in \g$ by
$L_{\diamond}(x)(y)= x\diamond y.$ Then  $(\g, L_{\diamond})$ is a representation of
$\mathfrak{g}$.
\end{proposition}
\begin{proof}
Since $\mathfrak{g}$ is JJ algebra,
 the Jacobi condition on $x,y,z\in \g$ is
 $$ x\diamond(y\diamond z) + y\diamond (z\diamond x)+ z\diamond(x\diamond y)=0
 $$
 and may be written
  $$L_{\diamond}(x\diamond y)(z)=-L_{\diamond}(x)(L_{\diamond}(y)(z))-L_{\diamond}(y)(L_{\diamond}(x)(z))
 $$
 Then the operator $\ad$ satisfies
 $$L_{\diamond}[x,y]=-L_{\diamond}(x)\circ L_{\diamond}(y)-L_{\diamond}(y)\circ L_{\diamond}(x).
 $$
 Therefore, it determines a representation of the JJ algebra
 $\mathfrak{g}$.
\end{proof}
Let $(\g, \diamond)$ be a
JJ algebra and  $(V,\rho)$ be a representation of
$\mathfrak{g}$. Let $V^*$ be the dual vector space of $V$.
We define a linear map 
\beq\label{dual1}
\rho^*: g & \longrightarrow & \mathfrak{gl}(V^*)  \cr
  x  & \longmapsto & \rho^*_x:
      \begin{array}{llll}
 V^* &\longrightarrow & V^* \\ 
  u^* & \longmapsto & \rho^*_x u^*: 
      \begin{array}{llll}
V  &\longrightarrow&  \K \cr
v  &\longmapsto& \left< \rho^*_xu^{*}, v \right>
 := \left<  u^{*}, \rho_x v\right>. 
      \end{array}
     \end{array}
 \eeq 

Let $f\in V^*$, $x,y\in \g$ and $u\in V$.
 We have
\begin{eqnarray*}
\left< \rho^*(x\diamond y) f, u\right>&=&
\left<  f,\rho(x\diamond y) u\right>\\
\ &=&\left<  f,-\rho(x)\rho(y) u-\rho(y)\rho(x) u\right>\\
\ &=&\left<  -\rho^*(x)\rho^*(y) f-\rho^*(y)\rho^*(x)f, u\right>.
\end{eqnarray*}
Therefore,
We have the following proposition:

\begin{proposition}
Let  $(\g, \diamond)$ be a
JJ algebra and  $(V,\rho)$ be a representation of
$\mathfrak{g}$, where $V$ is a finite dimensional vector space.
The following conditions are equivalent:
\begin{enumerate}
\item $(V,\rho)$ is a bimodule of $\g$.
\item $(V^{*},\rho^*)$ is a bimodule of
 $\g .$  
\end{enumerate}
\end{proposition}
\begin{corollary}
Let $(\g,\diamond)$ be a
JJ algebra and $(\g, L_{\diamond})$ be the adjoint representation of
 $\mathfrak{g}$, where $L_{\diamond}: \g\rightarrow End(\g)$.
 We set $L_{\diamond}: \g\rightarrow End(\g^*)$ and
 \beq\label{dual1}
L^*_{\diamond}: g & \longrightarrow & \mathfrak{gl}(V^*)  \cr
  x  & \longmapsto & L^*_{\diamond}(x):
      \begin{array}{llll}
 V^* &\longrightarrow & V^* \\ 
  u^* & \longmapsto & L^*_{\diamond}(x) u^*: 
      \begin{array}{llll}
V  &\longrightarrow&  \K \cr
f  &\longmapsto& \left< L^*_{\diamond}(x)u^{*}, v \right>
 := \left<  u^{*}, L_{\diamond}(x)v\right>. 
      \end{array}
     \end{array}
 \eeq 

 Then $(\g^*,L^*_{\diamond})$
 is a representation of $\mathfrak{g}$.
\end{corollary}

 \begin{theorem}
\label{dfnmatched} \label{matchJacobyJordan}
Let $(\displaystyle \mathcal{G}, \diamond)$ and $(\displaystyle \mathcal{H}, \bullet)$ be two JJ algebras and let  
$\displaystyle \mu: \mathcal{H} \rightarrow \mathfrak{gl}(\mathcal{G})$ and $\displaystyle \rho: \mathcal{G}\rightarrow \mathfrak{gl}( \mathcal{H})$
be two JJ algebra representations. Then, $\displaystyle(\mathcal{G}, \mathcal{H}, \rho, \mu)$ is called a matched pair of the JJ algebras
$\displaystyle \mathcal{G}$ and $\displaystyle \mathcal{H},$ 
denoted by $\displaystyle \mathcal{H}\bowtie_{\mu}^{-1,\rho}\mathcal{G}$ if and only if $\mu $ and $\rho$ satisfy:
$\displaystyle \mbox{ for all } x, y \in \mathcal{G}, a, b \in~\mathcal{H},$ 
\begin{eqnarray}\label{eqt1}
\rho(x)(a\bullet b)+ \rho(x)a\bullet b + a\bullet\rho(x)b +\rho(\mu(a)x)b+\rho(\mu(b)x)a=0,
\end{eqnarray}
\begin{eqnarray}\label{eqt2}
\mu(a)(x\diamond y)+ \mu(a)x\diamond y + x\diamond\mu(a)y + \mu(\rho(x)a)y+\mu(\rho(y)a)x=0.
\end{eqnarray}
In this case,  $\displaystyle (\mathcal{G}\oplus \mathcal{H}, \ast)$ defines a JJ algebra with respect to the product $\ast$ satisfying: 
\beq
\displaystyle (x+a)\ast(y+b)= x\diamond y + \mu(a)y+\mu(b)x + a\bullet b + \rho(x)b+\rho(y)a.
\eeq
\end{theorem}

\textbf{Proof:}
We have:
\beq
&&((x+a)\ast(y+b))\ast(z+c)+((y+b)\ast(z+c))\ast(x+a) + ((z+c)\ast(x+a))\ast(y+b)\cr 
 &=& \{(x\diamond y)\diamond z + (y\diamond z)\diamond x + (z\diamond x)\diamond y\}
+\{(a\bullet b)\bullet c + (b\bullet c)\bullet a + (c\bullet a)\bullet b\} \cr
&+&\{\mu(c)(x\diamond y)+ (\mu(c)x\diamond y)+ x\diamond\mu(c)y + \mu(\rho(x)c)y 
+\mu(\rho(y)c)x \} \cr 
&+&\{\rho(z)(a\bullet b) + \rho(z)a\bullet b + a\bullet\rho(z)b +\rho(\mu(a)z)b+\rho(\mu(b)z)a\}\cr 
& +&\{\mu(a)(y\diamond z)+ \mu(a)y\diamond z + y\diamond\mu(a)z + \mu(\rho(y)a)z+\mu(\rho(z)a)y\}\cr 
& +&\{\rho(x)(b\bullet c) + \rho(x)b\bullet c + b\bullet\rho(x)c + \rho(\mu(b)x)c+ \rho(\mu(c)x)b\} \cr
&+&\{\mu(b)(z\diamond x)+ \mu(b)z\diamond x + z\diamond\mu(b)x + \mu(\rho(z)b)x+\mu(\rho(x)b)z\} \cr 
&+&\{\rho(y)(c\bullet a) + \rho(y)c\bullet a + c\bullet\rho(y)a + \rho(\mu(c)y)a+ \rho(\mu(a)y)c\} \cr
&+&\{\rho(x\diamond y)c + \rho(x)(\rho(y)c)+ \rho(y)(\rho(x)c)\} 
+\{\rho(y\diamond z)a + \rho(y)(\rho(z)a)+\rho(z)(\rho(y)a)\} \cr
&+&\{\mu(a\bullet b)z+\mu(a)(\mu(b)z)+\mu(b)(\mu(a)z)\} +\{\mu(b\bullet c)x+ \mu(b)(\mu(c)x)+\mu(c)(\mu(b)x)\} \cr
&+&
+\{\mu(c\bullet a)+\mu(c)(\mu(a)y)+\mu(a)(\mu(c)y)\} +\{\rho(z\diamond x)b+ \rho(z)(\rho(x)b)+\rho(x)(\rho(z)b)\}.
\eeq
Thus setting in this previous computation $\displaystyle \mbox{ for all } x, y \in \mathcal{G}, a, b \in~\mathcal{H},$ 
\beqs
\rho(x)(a\bullet b) + \rho(x)a\bullet b + a\bullet \rho(x)b + \rho(\mu(a)x)b+\rho(\mu(b)x)a=0,
\eeqs
\beqs
\mu(a)(x\diamond y)+ \mu(a)x\diamond y + x\diamond \mu(a)y + \mu(\rho(x)a)y+\mu(\rho(y)a)x=0,
\eeqs
and using the linearity of the representations $\mu$ and $\rho$, the Jacobi identity
\beqs
((x+a)\ast(y+b))\ast(z+c)+ ((y+b)\ast (z+c))\ast(x+a) + ((z+c)\ast (x+a))\ast (y+b)=0
\eeqs
is satisfied. This end the proof.

$\cqfd$
 \subsection{Definition and main properties of pre-Jacobi-Jordan algebras}
 \begin{definition}\label{def}
 $(\A, \cdot),$ (or simply $\A$), is said to be a left pre-JJ-algebra 
  if $\forall \, x, y, z \, \in \A,$  the antiassociator of the bilinear product $\cdot$ defined by 
  $(~x,~y,~z~)_{-1}:= (~x\cdot y~)\cdot z +x\cdot~(~y\cdot z~),$ is symmetric in $x$ and $y$, i.e.,
 \begin{eqnarray}\label{CSA}
 (~x,~y,~z~)_{-1} = -(~y,~x,~z~)_{-1}.
 \end{eqnarray}
 \end{definition} 
As matter of  notation simplification, we will denote $\displaystyle x \cdot y$ by $xy$ if not any confusion.
\begin{proposition}
Any antiassociative algebra is a left pre-JJ-algebra, and
\beq\label{oper}
\mu\circ(\mu\otimes id)+\mu\circ(id\otimes\mu)+\mu\circ((\mu\circ\tau)\otimes id)+\mu(id\otimes (\mu\circ\tau))=0,
\eeq
where $\mu $ is the multiplication operator ant $\tau$ is such that $\tau(x\otimes y)=y\otimes x$.
\end{proposition}
 \begin{proof}
Let $(\A, \cdot)$ be an antiassociative algebra. We have $\forall x, y,z\in \A$
\beqs
&&(x\cdot y)\cdot z= -x\cdot(y\cdot z)\cr&\Leftrightarrow &(x,y,z)_{-1}=0=-(y,x,z)_{-1}.
\eeqs
Thus $(x,y,z)_{-1}+(y,x,z)_{-1}=0$, which implies that 
\beqs
(xy)z+x(yz)+(yx)z+y(xz)=0.
\eeqs
 Supposing  $\mu$ be the bilinear product on $\A$, then the previous relation can be written 
 as
 \beq\label{eq_CSA}
 &&\mu(\mu(x\otimes y)\otimes z)+\mu(x\otimes\mu(y\otimes z))+\mu((\mu\circ\tau)(x\otimes y)\otimes z)+\mu(y\otimes(\mu\circ\tau)z\otimes x)=0\cr
 &\Leftrightarrow&
 \mu\circ(\mu\otimes id)+\mu\circ(id\otimes\mu)+\mu\circ((\mu\circ\tau)\otimes id)+\mu(id\otimes (\mu\circ\tau))=0
 \eeq
 \end{proof}
 \begin{definition}
An algebra $(\A, \cdot)$ over $\K$ with the bilinear product given by 
$(x,y)\mapsto x\cdot y$ is called right pre-JJ algebra if the associator associated 
to the bilinear product on $\A$ is symmetric in right, i.e.  for all  $x,y,z\in \A$
\beq\label{eq_right_sym}
(x,y,z)_{-1}=-(x,z,y)_{-1}.
\eeq
\end{definition}
\begin{proposition}\label{thm_equiv_left_right}
The opposite algebra of a left pre-JJ algebra is a right pre-JJ algebra. 
\end{proposition}
\textbf{Proof:}

Suppose that $(\A, \cdot)$ is a pre-JJ algebra and  $(\A, \circ )$ the opposite algebra 
of the algebra $\A$. For any $x, y, z\in \A$ we have:
\beqs
(x,y,z)_{-1,\circ}&=&(x\circ y)\circ z+x\circ (y\circ z)
=z\cdot (y\cdot x)+(z\cdot y)\cdot x
= (z, y,x)_{-1}=-(y, z, x)_{-1} \cr &=& 
-y\cdot (z\cdot x)-(y\cdot z)\cdot x=-(x\circ z)\circ y-x\circ (z\circ y)=-(x, z,y)_{-1,\circ}.
\eeqs
Therefore, for all $x, y,z\in \A$,  $(x,y,z)_{-1,\circ}=-(x ,z,y)_{-1,\circ}$.  
$\cqfd$
 
 In the following, left pre-JJ algebras and right pre-JJ algebras are equivalent. Any left(right) pre-JJ algebra will be right(left) pre-JJ algebra under new multiplication $(x,y)\to y\cdot x$. Thus for the rest of this paper, without any further clarification, left pre-JJ algebra are called pre-JJ algebra. In addition,  to have an ease in manipulations, we replace $\diamond$ by $"[.,.]".$

 Considering the representations of the  left $L$ and right $R$ multiplication operations:
 \begin{eqnarray}
 L: \A & \longrightarrow & \mathfrak{gl}(\A)  \cr
  x  & \longmapsto & L_x:
  \begin{array}{ccc}
 \A &\longrightarrow & \A \cr 
  y & \longmapsto & x \cdot y, 
   \end{array}
\end{eqnarray}
\begin{eqnarray}
    R: \A & \longrightarrow & \mathfrak{gl}(\A)  \cr
     x  & \longmapsto & R_x:
     \begin{array}{ccc}
    \A &\longrightarrow & \A \cr 
     y & \longmapsto & y \cdot x,
      \end{array}
 \end{eqnarray}
 we infer the adjoint representation  $\ad: = L+R$ of the sub-adjacent JJ algebra  of a pre-JJ algebra $\A$  as follows:
  \begin{eqnarray}
    \ad: \A & \longrightarrow & \mathfrak{gl}(\A)  \cr
      x  & \longmapsto & \ad_x:
      \begin{array}{ccc}
     \A &\longrightarrow & \A \cr 
      y & \longmapsto & \ad_x(y),
       \end{array}
    \end{eqnarray}
such that  
  $\displaystyle \forall \; x, y \in \A, \ad_x(y):=(L_x+R_x)(y).$ 
\begin{proposition}\label{one}
Let $(\A, \cdot )$ be a pre-JJ algebra. For any $x,y\in \A$, the following relations are satisfied:
\begin{itemize}
\item The anticommutator associated to the bilinear product $"\cdot"$ given by
$[x,y]=x\cdot y+y\cdot x$ defines a JJ algebra structure on $\A$.
\item
The left multiplication operator gives a representation of the JJ algebra, that is 
\beqs
L_{[x,y]}=-[L_x, L_y],
\eeqs
and the following relation is also satisfied
\beqs
[L_x, R_y]=-R_{x\cdot y}-R_yR_x,
\eeqs
where the linear map $R$ is the right multiplication operator associated to the bilinear product $"\cdot"$ on 
$\A$.
\item $[L_x,R_y]=-[R_x,L_y]$
\item $L_{xy}+L_xL_y=-L_{yx}-L_yL_x$,
\item $ad=L+R$ a linear representation of the sub-adjacent JJ algebra of $(\A, \cdot )$ and,
\beqs
[ad_x,ad_y]=ad_{_{[x,y]}}.
\eeqs
\end{itemize}
\end{proposition}
\textbf{Proof:}

Consider the pre-JJ algebra $(\A, \cdot)$. For any $x,y,z\in \A$ we have
\beqs
[x,[y,z]]+[y, [z,x]]+[z, [y, x]]&=&[x, yz+zy]+[y, zx+xz]+[z, xy+yx]\cr 
&=& x(yz)+x(zy)+(yz)x+(zy)x\cr
&+& y(zx)+y(xz)+(zx)y+(xz)y\cr 
&+& z(xy)+z(yx)+(xy)z+(yx)z\cr
&=&\{x(yz)+(yz)x+ y(zx)+(zx)y+ z(xy)+(xy)z  \}\cr
&+& \{ (zy)x+x(zy)+y(xz)+(xz)y+z(yx)+(yx)z\}\cr 
&=& \{(x,y,z)_{_{-1}}+(y,z,x)_{_{-1}} +(z, x, y)_{_{-1}} \}\cr 
&+& \{(z,y,x)_{_{-1}}+(x,z,y)_{_{-1}}+(y, x,z)_{_{-1}} \}\cr 
&=& \{(y, x,z)_{_{-1}}+(x,y,z)_{_{-1}}\}+\{ (z,y,x)_{_{-1}}+(y,z,x)_{_{-1}}\}\cr 
&+&\{(x,z,y)_{_{-1}} +(z, x, y)_{_{-1}}\}\cr 
&=& 0.
\eeqs
On other hand, we have for all  $x, y, z\in \A$ 
\beqs
(x,y,z)_{_{-1}}=-(y,x,z)_{_{-1}} &\Longleftrightarrow &
(xy)z+x(yz)=-(yx)z-y(xz)\cr 
&\Longleftrightarrow &
(xy)z+(yx)z=-x(yz)-y(xz)\cr 
&\Longleftrightarrow &
(L_{xy}+L_{yx})(z)=-(L_xL_y+L_yL_x)(z)\cr 
&\Longleftrightarrow &
(L_{xy}+L_{yx})(z)=-[L_x,L_y](z)\cr 
\eeqs
Therefore, the relation for all  $x,y\in \A$
\beqs
L_{[x,y]}=-[L_x,L_y]
\eeqs
 holds. 
 
We have for all  $x,y,z\in \A$
\beqs
(x,y,z)_{_{-1}}=-(y,x,z)_{_{-1}} &\Longleftrightarrow &
(xy)z+x(yz)=-(yx)z-y(xz)\cr 
&\Longleftrightarrow &
(R_zL_x-L_xR_z)(y)=-(R_zR_x+R_{xz})(y)\cr 
&\Longleftrightarrow &
[R_z, L_x](y)=-(R_zR_x+R_{xz})(y)
\eeqs
Therefore, the following relation holds
\beqs
[L_x, R_y]=-(R_{xy}+R_yR_x).
\eeqs
We have for all $x,y,z\in\A$
\begin{align*}
[L_x,R_y](z)&=L_x\left(R_y(z)\right)+R_y\left(L_x(z)\right)\\
&=x(zy)+(xz)y\\
&=(x,z,y)_{_{-1}}=-(z,x,y)_{_{-1}}\\
&=-\left((zx)y+z(xy)\right)=-\left(R_y\left(R_x(z)\right)+R_{xy}(z)\right)\\
&=-\left(R_yR_x+R_{xy}\right)(z)\\
&=(z,y,x)_{_{-1}}=(zy)x+z(yx)=R_xR_y(z)+R_{xy}(z)\\
&=-(y,z,x)-(yz)x-y(zx)\\
&=-R_xL_y(z)-L_yR_x(z)\\
&=-[R_x,L_y](z).
\end{align*}
Therefore $[L_x,R_y]=-[R_x,L_y]$.\\
We have for all $x,y,z\in\A$
\begin{align*}
\left(L_{xy}+L_xL_y\right)(z)&=L_{xy}(z)+L_x\left(L_y(z)\right)=(xy)z+x(yz)\\
&=(x,y,z)_{_{-1}}=-(y,x,z)_{_{-1}}=-\left((yx)z+y(xz)\right)\\
&=-\left(L_{yx}(z)+L_y(L_x(z))\right)\\
&=-\left(L_{yx}+L_yL_x\right)(z).
\end{align*}
Thus $L_{xy}+L_xL_y=-L_{yx}-L_yL_x$.\\
We have for all $x,y\in\A$,
\begin{align*}
[ad_x,ad_y]&=[L_x,L_y]+[L_x,R_y]+[R_x,L_y]+[R_x,R_y]\\
&=[L_x,L_y]+[R_x,R_y]\\
&=\left(L_xL_y+R_xR_y\right)+\left(L_yL_x+R_yR_x\right)\\
&=(L+R)_{_{xy}}+(L+R)_{_{yx}}\\
&=(L+R)_{_{xy+yx}}\\
&=(L+R)_{_{[x,y]}}\\
&=ad_{_{[x,y]}}
\end{align*}
$\cqfd$\\


 Thus, $(\A, \cdot)$ can be called the compatible pre-JJ algebra product of the JJ algebra $\G(\A).$
 \section{Bimodules and matched pairs of pre-Jacobi-Jordan algebras}
 \begin{definition}\label{bimodule}
Let $\A$ be a pre-JJ algebra, $\displaystyle V$ be a vector space. Suppose  
$\displaystyle l,r : \A \rightarrow \mathfrak{gl}(V)$ be two linear maps satisfying: 
$\mbox{for all } x, y \in \A,$
\begin{eqnarray}\label{eqbimodule1}
\left[l_x, r_y\right]=- \left[l_y, r_x\right]
\end{eqnarray}
\begin{eqnarray}\label{eqbimodule2}
l_{xy}+l_xl_y=-l_{yx}-l_yl_{x}.
\end{eqnarray}
Then, $\displaystyle(l, r, V)$ (or simply $\displaystyle(l,r)$) is called bimodule of the pre-JJ algebra~$\A.$
 \end{definition}
\begin{proposition}
Let $(\A, \cdot)$ be an pre-JJ algebra and $V$ be a vector space over $\K.$
Consider two linear maps,
 $\displaystyle l, r : \A \rightarrow \mathfrak{gl}(V).$ Then, $(l, r, V)$ is a bimodule of $\A$ if and only if, 
 the semi-direct sum  $\A\oplus V $ of vector spaces is turned into a pre-JJ algebra  by defining the multiplication in $\A\oplus V$ by $ \forall x_1, x_2 \in \A, \; v_1, v_2 \in V$,
\begin{eqnarray*}
(x_1+v_1)\ast (x_2+v_2) = x_1\cdot x_2+(l_{x_1}v_2+r_{x_2}v_1), 
\end{eqnarray*} 
We denote it by $\displaystyle \A \ltimes_{l, r}^{-1} V $ or simply $\A \ltimes^{-1}  V.$
\end{proposition} 
\textbf{Proof:}

It is obvious that the semi-direct sum of two vector spaces is also a vector space. Now
suppose that $(l,r, V)$ is a bimodule of $\displaystyle \A$ and
show that $\displaystyle (\A\oplus V, \ast )$ is a pre-JJ algebra. Since $\displaystyle \ast $ is a bilinear product,

for all $x_1, x_2,x_3 \in \A$ and for all $v_1, v_2, v_3\in V$, we have:

\beqs
\left((x_1+v_1), (x_2+v_2), (x_3+v_3)\right)_{_{-1}}
&=& \{(x_1+v_1)\ast (x_2+v_2)\}\ast (x_3+v_3)\\
&+&(x_1+v_1)\ast \{(x_2+v_2)\ast (x_3+v_3)\}\\
&=& (x_1x_2)x_3+l_{x_1x_2}v_3+r_{x_3}(l_{x_1}v_2+
r_{x_2}v_1)\\
&+& x_1(x_2x_3)+l_{x_1}(l_{x_2}v_3)+
l_{x_1}(r_{x_3}v_2)+r_{x_2x_3}v_1\\
&=& (x_1x_2)x_3+l_{x_1x_2}v_3+r_{x_3}(l_{x_3}v_2)+r_{x_3}(r_{x_2}v_1)\\
&+&  x_1(x_2x_3)+l_{x_1}(l_{x_2}v_3)+
l_{x_1}(r_{x_3}v_2)+r_{x_2x_3}v_1\\
\left((x_1+v_1), (x_2+v_2), (x_3+v_3)\right)
&=& (x_1, x_2, x_3)_{_{-1}}+(l_{x_1x_2}+l_{x_1}l_{x_2})v_3\\
&+&
(r_{x_3}l_{x_1}+l_{x_1}r_{x_3})v_2+
(r_{x_3}r_{x_2}+r_{x_2x_3})v_1.
\eeqs
\beqs
\left((x_2+v_2), (x_1+v_1), (x_3+v_3)\right)
&=& (x_2, x_1, x_3)+(l_{x_2x_1}+l_{x_2}l_{x_1})v_3\\
&+&
(r_{x_3}l_{x_2}+l_{x_2}r_{x_3})v_1+
(r_{x_3}r_{x_1}+r_{x_1x_3})v_2.
\eeqs
Therefore, 
\beqs
(x_1+v_1, x_2+v_2, x_3+v_3)_{_{-1}}
&=&-(x_3+v_3, x_2+v_2, x_1+v_1)_{_{-1}} \\
&\Longleftrightarrow &
\left\lbrace
\begin{array}{ccc}
& l_{x_1x_2}+l_{x_1}l_{x_2}=
-l_{x_2x_1}-l_{x_2}l_{x_1}\\
& r_{x_3}l_{x_1}+l_{x_1}r_{x_3}=
-r_{x_3}r_{x_1}-r_{x_1x_3} \\
& r_{x_2x_3}+r_{x_3}r_{x_2}=
-r_{x_3}l_{x_2}-l_{x_2}r_{x_3}
\end{array}
\right.\\
& \Longleftrightarrow & (\A\oplus V, \ast ) \mbox{ is a pre-JJ algebra.}
\cqfd
\eeqs
Furthermore, we derive the next result.
\begin{proposition}\label{propcentlie}
Let  $\A$ be a pre-JJ algebra and $V$ be a vector space over $\K$.
Consider two linear maps,
 $\displaystyle l, r: \A \rightarrow \mathfrak{gl}(V),$ 
 such that 
$(l,r, V)$ is a bimodule of $\A.$ Then, the map:
$ \displaystyle 
 l+r: \A  \longrightarrow  \mathfrak{gl}(V) \;  
  x   \longmapsto  l_x+r_x,
  $
is a linear representation of the sub-adjacent JJ algebra of $\A.$
\end{proposition}
\textbf{Proof:}

Let $\displaystyle (l, r, V)$ be a bimodule of the pre-JJ algebra $\A.$ Then, $\displaystyle \forall x, y \in \A$\, \, 
$\displaystyle [l_x, r_y]=-[l_y, r_x]; l_{xy}+l_xl_y=-l_yl_x-l_{yx}$. 
Besides, it is a matter of straightforward computation to show that $l+r$ is a linear map on $\A.$ Then, we have:
\beqs
[(l+r)(x), (l+r)(y)] 
&=& [l_x+r_x, l_y+r_y] 
=[l_x, l_y]+[l_x, r_y]+[r_x, l_y]+[r_x, r_y] \\
&=& [l_x, l_y]+[r_x, r_y] \\
&=& 
l_xl_y+l_yl_x+r_xr_y+r_yr_x 
= \{l_xl_y+r_xr_y\}+\{l_yl_x+r_yr_x\} \\
&=& 
\{l_{xy}+r_{xy}\}+\{l_{yx}+r_{yx}\} 
=(l+r)_{xy}+(l+r)_{yx} 
=(l+r)_{[x, y]}.
\eeqs
Therefore, $\displaystyle (l, r, V)$ is a bimodule of $\displaystyle \A$ implies that $\displaystyle l+r$ 
is a representation of the linear representation of the sub-adjacent JJ algebra of $\A.$
$\cqfd$
\begin{example}
According to the Proposition~\ref{one}, one can deduce that $(L, R, \A)$ is a bimodule of the pre-JJ algebra $\A$, 
where $L$ and $R$ are the left and right multiplication operator representations, respectively.
\end{example}

 \begin{theorem}\label{theoo}
 Let $(\A, \cdot)$ and $(\B, \circ)$ be two pre-JJ  algebras. 
Suppose that $(l_{\A}, r_{\A}, \B)$ and $(l_{\B}, r_{\B}, \A)$
 are bimodules of $\A$ and $\B$, respectively, obeying the relations:
\begin{eqnarray}\label{eqq1}
  r_{\mathcal{A}}(x)([a,b]) = r_{\mathcal{A}}(l_{\mathcal{B}}(b)x)a+r_{\mathcal{A}}(l_{\mathcal{B}}(a)x)b+
a \circ (r_{\mathcal{A}}(x)b)+b \circ (r_{\mathcal{A}}(x)a),
\end{eqnarray}
\begin{eqnarray}\label{eqq2}
  -l_{\mathcal{A}}(x)(a \circ b) &=& l_{\mathcal{A}}(l_{\mathcal{B}}(a)x+r_{\mathcal{B}}(a)x)b +
(l_{\mathcal{A}}(x)a+r_{\mathcal{A}}(x)a)\circ b \cr
&+&r_{\mathcal{A}}(r_{\mathcal{B}}(b)x)a+a\circ (l_{\mathcal{A}}(x)b), 
\end{eqnarray}
\begin{eqnarray}\label{eqq3}
 r_{\mathcal{A}}(a)[x,y]=r_{\mathcal{B}}(l_{\mathcal{A}}(y)a)x+r_{\mathcal{B}}(l_{\mathcal{A}}(x)a)y+
x(r_{\mathcal{B}}(a)y)+y(r_{\mathcal{B}}(a)x), 
 \end{eqnarray}
\begin{eqnarray}\label{eqq4}
 -l_{\mathcal{B}}(a)(xy)&=&
l_{\mathcal{B}}(l_{\mathcal{A}}(x)a)y+(r_{\mathcal{B}}(a)x)y+x(l_{\mathcal{B}}(a)y)+r_{\mathcal{B}}(r_{\mathcal{A}}(y)a)x \cr
&+&(l_{\mathcal{B}}(a)x)y+l_{\mathcal{B}}(r_{\mathcal{A}}(x)a)y +l_{\mathcal{B}}(a)(xy),
 \end{eqnarray}
 for all  $x, y \in A$ and $a, b \in \B.$ Then, there is a pre-JJ algebra structure on $\A \oplus \B$ 
 given by:
 \beq\label{pro} 
 \displaystyle (x+a)\ast (y+b)= (x \cdot y + l_{\B}(a)y+r_{\B}(b)x)+ (a \circ b + l_{\A}(x)b+r_{\A}(y)a).
 \eeq
 We denote this pre-JJ algebra by $\displaystyle \A \bowtie_{l_{\B}, r_{\B}}^{-1, l_{\A}, r_{\A}} \B,$ or
 simply by $\displaystyle \A \bowtie^{-1} \B.$ Then $(\A, \B, l_{\A}, r_{\A}, l_{\B}, r_{\B})$ satisfying
the above conditions is called  matched pair of the pre-JJ algebras $\A$ and $\B.$
 \end{theorem}
\textbf{Proof:}

Consider $x, y \in \mathcal{A}$ and $a, b \in \mathcal{B}$. We have\\ 
\beqs 
(x+a) \ast (y+b) &=& (x y+ l_{\mathcal{B}}(a)y+r_{\mathcal{B}}(b)x)+ (a \circ b + l_{\mathcal{A}}(x)b+r_{\mathcal{A}}(y)a), \\
x \ast a &=& r_{\mathcal{B}}(a)x + l_{\mathcal{A}}(x)a,\,\,  b=0, y=0,\\
a \ast y &=&  l_{\mathcal{B}}(a)x + r_{\mathcal{A}}(y)a ,\,\, x=0, b=0.
\eeqs
Since $
(x,y,z)_{_{-1}}= (xy)z+x(yz)$, we have
\beqs
 (x,b,c)_{_{-1}} &=& (x \ast b)\ast c + x\ast (b \circ c) 
= 
(r_{\mathcal{B}}(b)x + l_{\mathcal{A}}(x)x)\ast c + x \ast (b \circ c) \cr
 &=& r_{\mathcal{B}}(c)(r_{\mathcal{B}}(b)x) + (l_{\mathcal{A}}(x)b)\circ c + 
l_{\mathcal{A}}(r_{\mathcal{B}}(b)x)c + r_{\mathcal{B}}(b \circ c)x + 
 l_{\mathcal{A}}(x)(b \circ c),
\eeqs
\beqs
(a,y,c)_{_{-1}} &=& (a \ast y)\ast c + a \ast (y \ast c) \cr
&=& (l_{\mathcal{B}}(a)y + r_{\mathcal{A}}(y)a)\ast c + a  \ast (l_{\mathcal{A}}(y)c + r_{\mathcal{B}}(c)y) \cr
&=& l_{\mathcal{A}}(l_{\mathcal{B}}(a)y)c + r_{\mathcal{B}}(c)(l_{\mathcal{B}}(a)y) + (r_{\mathcal{A}}(y)a)\circ c \cr
&+& a \circ (l_{\mathcal{A}}(y)c) + l_{\mathcal{B}}(a)(r_{\mathcal{B}}(c)y) + r_{\mathcal{A}}(r_{\mathcal{B}}(c)y)a,
\eeqs
\beqs
(x,y,c)_{_{-1}} &=& (x\cdot y)\ast c + x \ast(y \ast c)  \cr
&=& (l_{\mathcal{A}}(x\cdot y)c +r_{\mathcal{B}}(c)(x \cdot y)\ast c + x
 \ast (l_{\mathcal{A}}(y)c+r_{\mathcal{B}}(c)y) \cr
&=& l_{\mathcal{A}}(xy)c+r_{\mathcal{B}}(c)(xy)+x\cdot (r_{\mathcal{B}}(c)y)+
l_{\mathcal{A}}(x)(l_{\mathcal{A}}(y)c)+r_{\mathcal{B}}(l_{\mathcal{A}}(y)c)x,
\eeqs
\beqs
(x,b,z)_{_{-1}} &=& (x \ast b)\ast z + x \ast (b \ast z) \cr
&=&(l_{\mathcal{A}}(x)b+r_{\mathcal{B}}(b)x)\ast z + x \ast (l_{\mathcal{B}}(b)z+r_{\mathcal{A}}(z)b)\cr
&=& l_{\mathcal{B}}(l_{\mathcal{A}}(x)b)z+r_{\mathcal{A}}(z)(l_{\mathcal{A}}(x)b)+
(r_{\mathcal{B}}(b)x)\cdot z \cr
&+& x \cdot (l_{\mathcal{B}}(b)z) +l_{\mathcal{A}}(x)(r_{\mathcal{A}}(z)b)+ r_{\mathcal{B}}(r_{\mathcal{A}}(z)b)x,
\eeqs
\beqs
(a,b,c)_{_{-1}}= (a \circ b)\circ c + a \circ (b \circ c)
\eeqs
\beqs
(a,y,z)_{_{-1}} &=& (a \ast y) \ast z + a \ast (y \cdot z) \cr
&=& (l_{\mathcal{B}}(a)y+r_{\mathcal{A}}(y)a)\ast z + l_{\mathcal{B}}(a)(y\cdot z)+r_{\mathcal{A}}(y\cdot z)a \cr
&=& (l_{\mathcal{B}}(a)y)z + l_{\mathcal{B}}(r_{\mathcal{A}}(y)a)z+ r_{\mathcal{A}}(z)(r_{\mathcal{A}}(y)a)
+ 
 l_{\mathcal{B}}(a)(y \cdot z)+r_{\mathcal{A}}(y \cdot z)a,
\eeqs
\beqs
(a,b,z)_{_{-1}} &=& (a \circ b) \ast z + a \ast (b \ast z) \cr
&=& (l_{\mathcal{B}}(a \circ a)\cdot z +r_{\mathcal{A}}(z)(a \circ b)) + a \ast (l_{\mathcal{B}}(b)z + 
r_{\mathcal{A}}(z)b) \cr
&=& l_{\mathcal{B}}(a \circ b)z+ r_{\mathcal{A}}(z)( \circ b)+ a \circ (r_{\mathcal{A}}(z)b)
+ 
 l_{\mathcal{B}}(a)(l_{\mathcal{B}}(b)z)-r_{\mathcal{A}}(l_{\mathcal{B}}(b)z)a,
\eeqs
\\
The first part of the associator reads:
\beqs
\{(x+a) \ast (y+b)\} \ast (z+c) &=& \{(xy+l_{\mathcal{B}}(a)y+r_{\mathcal{B}}(b)x) +  (a \circ b +l_{\mathcal{A}}(y)+r_{\mathcal{A}}(y)a)\} \ast (z+c) \cr
&=& (xy + l_{\mathcal{B}}(a)y+r_{\mathcal{B}}(b)x)z + l_{\mathcal{B}}(a \circ b + 
l_{\mathcal{A}}(x)b+ r_{\mathcal{A}}(y)a)z \cr
&+ & r_{\mathcal{B}}(c)(xy + l_{\mathcal{B}}(a)y+r_{\mathcal{B}}(b)x)+(a \circ b + 
l_{\mathcal{A}}(x)b+r_{\mathcal{A}}(y)a)\circ c \cr
& +& l_{\mathcal{A}}(a)(yz+l_{\mathcal{B}}(b)z+r_{\mathcal{B}}(c)y) + r_{\mathcal{A}}(z)(a \circ b + 
l_{\mathcal{A}}(x)b+r_{\mathcal{A}}(y)a) \cr
&=&(xy)z +(l_{\mathcal{B}}(a)y)z + (r_{\mathcal{B}}(b)x)z+ l_{\mathcal{B}}(a \circ b)z+ l_{\mathcal{B}}(l_{\mathcal{A}}(x)b)z\cr
&+ &l_{\mathcal{B}}(r_{\mathcal{A}}(y)a)z+ r_{\mathcal{B}}(c)(xy)+ r_{\mathcal{B}}(c)(l_{\mathcal{B}}(a)y)+
r_{\mathcal{B}}(c)(r_{\mathcal{B}}(b)x)) \cr
&+ &(a \circ b)\circ c +(l_{\mathcal{A}}(x)b)\circ c+(r_{\mathcal{A}}(y)a)\circ c+l_{\mathcal{A}}(a)(yz)\cr
&+ &l_{\mathcal{A}}(a)(l_{\mathcal{B}}(b)z)+l_{\mathcal{A}}(a)(r_{\mathcal{B}}(c)y)+r_{\mathcal{A}}(z)(a \circ b)\cr
&+ & r_{\mathcal{A}}(z)(l_{\mathcal{A}}(x)b)+r_{\mathcal{A}}(z)(r_{\mathcal{A}}(y)a)
\eeqs
while its second part :
\beqs
(x+a) \ast \{(y+b) \ast (z+c)\} &=& (x+a) \ast \{yz+l_{\mathcal{B}}(b)z+r_{\mathcal{B}}(c)y+ b \circ c + l_{\mathcal{A}}(y)c+ r_{\mathcal{A}}(z)b\} \cr
&=& x(yz+l_{\mathcal{B}}(b)z+r_{\mathcal{B}}(c)y)+l_{\mathcal{B}}(a)(yz+l_{\mathcal{B}}(b)z+r_{\mathcal{B}}(c)y) 
+r_{\mathcal{B}}(b \circ c \cr &+& l_{\mathcal{A}}(y)c+r_{\mathcal{A}}(z)b)x+ a \circ (b \circ c +l_{\mathcal{A}}(y)c 
 +r_{\mathcal{A}}(z)b)\cr&+&l_{\mathcal{A}}(x)(b \circ c + l_{\mathcal{A}}(y)c+r_{\mathcal{A}}(z)b)
+r_{\mathcal{A}}(yz+l_{\mathcal{B}}(b)z+r_{\mathcal{B}}(c)y)a \cr
&=&x(yz)+x(l_{\mathcal{B}}(b)z)\cr 
&+ &x(r_{\mathcal{B}}(c)y)+l_{\mathcal{B}}(a)(yz)+l_{\mathcal{B}}(a)(l_{\mathcal{B}}(b)z)
+l_{\mathcal{B}}(a)(r_{\mathcal{B}}(c)y)\cr&+& r_{\mathcal{B}}(b \circ c)x+ r_{\mathcal{B}}( l_{\mathcal{A}}(y)c)x +
r_{\mathcal{B}}(r_{\mathcal{A}}(z)b)x
+a\circ (b \circ c)\cr &+&a \circ (l_{\mathcal{A}}(y)c)+a \circ (r_{\mathcal{A}}(z)b)+l_{\mathcal{A}}(x)(b \circ c)
+l_{\mathcal{A}}(x)(l_{\mathcal{A}}(y)c)\cr&+&l_{\mathcal{A}}(x)(r_{\mathcal{A}}(z)b)+r_{\mathcal{A}}(yz)a+ r_{\mathcal{A}}(l_{\mathcal{B}}(b)z)a+r_{\mathcal{A}}(r_{\mathcal{B}}(c)y)a
\eeqs
and the associator takes the form:
\beqs
(x+a, y+b, z+c)_{_{-1}} &=&\{ (xy)z-x(yz) \} +\{ (a \circ b)\circ c - a \circ (b \circ c)\}\cr
&+&\{r_{\mathcal{B}}(c)(r_{\mathcal{B}}(b)x)+(l_{\mathcal{A}}(x)b)\circ c+l_{\mathcal{A}}(r_{\mathcal{B}}(b)x)c \cr
&-&r_{\mathcal{B}}(b \circ c)x-l_{\mathcal{A}}(x)(b \circ c)\} 
\{l_{\mathcal{A}}(l_{\mathcal{B}}()a)yc+r_{\mathcal{B}}(c)(l_{\mathcal{B}}(a)y)+(r_{\mathcal{A}}y)\circ c\cr 
&-&a\circ (l_{\mathcal{A}}(y)c)- l_{\mathcal{B}}(a)(r_{\mathcal{B}}(c)y)-r_{\mathcal{A}}(r_{\mathcal{B}}(c)y)a \}+ \{l_{\mathcal{A}}(l_{\mathcal{B}}(a)y)c\cr 
&+&
r_{\mathcal{B}}(c)(l_{\mathcal{B}}(a)y)+ (r_{\mathcal{A}}(y)a)\circ c-a \circ (l_{\mathcal{A}}(y)c)-l_{\mathcal{B}}(a)(r_{\mathcal{B}}(c)y)\cr &-&
r_{\mathcal{A}}(r_{\mathcal{B}}(c)y)a\}+ \{l_{\mathcal{B}}(l_{\mathcal{A}}(x)b)z+r_{\mathcal{A}}(z)(l_{\mathcal{A}}(x)b)+(r_{\mathcal{B}}(b)x)z\cr 
&-& x(l_{\mathcal{B}}(b)z)- l_{\mathcal{A}}(x)(r_{\mathcal{A}}(z)b)x\}+ \{(l_{\mathcal{B}}(a)y)z+l_{\mathcal{B}}(r_{\mathcal{A}}(y)a)z\cr 
&+& r_{\mathcal{A}}(z)(r_{\mathcal{A}}(y)a)-l_{\mathcal{B}}(a)(yz)-r_{\mathcal{A}}(yz)a\} + \{l_{\mathcal{B}}(a \circ b)z\cr 
&+& r_{\mathcal{A}}(z)(a \circ b)-a \circ (r_{\mathcal{A}}(z)b) -l_{\mathcal{B}}(a)(l_{\mathcal{B}}(b)z)-r_{\mathcal{A}}(l_{\mathcal{B}}(b)z)a\}\cr
&=&(x, y, z)_{_{-1}} + (a, b, c)_{_{-1}} + \{r_{\mathcal{B}}(c)(x \cdot y) + 
l_{\mathcal{A}}(x \cdot y)c +x \cdot (r_{\mathcal{B}}(c)y) 
\cr
&+&
l_{\mathcal{A}}(x)(l_{\mathcal{A}}(y)c)
+r_{\mathcal{B}}(l_{\mathcal{A}}(y)c)x \} + 
\{r_{\mathcal{B}}(c)(r_{\mathcal{B}}(b)x) 
+ l_{\mathcal{A}}(r_{\mathcal{B}}(b)x)c 
\cr  &+& 
r_{\mathcal{B}}(b \circ c)x + (l_{\mathcal{A}}(x)b)\circ c+
l_{\mathcal{A}}(x) ( b \circ c) \} 
+ \{(r_{\mathcal{B}}(b)x) \cdot z+ 
 l_{\mathcal{B}}(l_{\mathcal{A}}(x)b)z
 \cr &+&
  r_{\mathcal{A}}(z)(l_{\mathcal{A}}(x)b)+
x \cdot (l_{\mathcal{B}}(b)z)+ 
 r_{\mathcal{B}}(r_{\mathcal{A}}(z)b)x+l_{\mathcal{A}}(x)(r_{\mathcal{A}}(z)b)\}
\cr &+&
 \{(l_{\mathcal{B}}(a)y)\cdot z
 + 
l_{\mathcal{B}}(r_{\mathcal{A}}(y)a)z 
+r_{\mathcal{A}}(z)(r_{\mathcal{A}}(y)a)
+ l_{\mathcal{B}}(a)(y \cdot z)\cr&+& r_{\mathcal{A}}(y \cdot z )a\}
+
\{r_{\mathcal{B}}(c)(l_{\mathcal{B}}(a)y) 
+(r_{\mathcal{A}}(y)a) \circ c +
 l_{\mathcal{A}}(l_{\mathcal{B}}(a)y)c
 \cr &+&
l_{\mathcal{B}}(a)(r_{\mathcal{B}}(c)y)
+ a \circ (l_{\mathcal{A}}(y)c) 
+r_{\mathcal{A}}(r_{\mathcal{B}}(c)y)a\} + \{l_{\mathcal{B}}(a \circ b)z \cr 
&+& r_{\mathcal{A}}(z)(a \circ b)+
 l_{\mathcal{B}}(a)(l_{\mathcal{B}}(b)z) 
+ a \circ (r_{\mathcal{A}}(z)b)+ r_{\mathcal{A}}(l_{\mathcal{B}}(b)z)a\}.
\eeqs
Further, we have 
\beqs
(x, y, c)_{_{-1}} &=& r_{\mathcal{B}}(c)(x \cdot y)+x \cdot (r_{\mathcal{B}}(c)y)+ r_{\mathcal{B}}(l_{\mathcal{A}}(y)c)x+l_{\mathcal{A}}(x \cdot y)c 
+ l_{\mathcal{A}}(x)(l_{\mathcal{A}}(y)c),\cr
(x, b, c)_{_{-1}} &=& r_{\mathcal{B}}(c)(r_{\mathcal{B}}(b)x)+l_{\mathcal{A}}(r_{\mathcal{B}}(b)x)c+ (l_{\mathcal{A}}(x)b)\circ c+ r_{\mathcal{B}}(b \circ c)x 
+l_{\mathcal{A}}(x)(b \circ c),\cr
(x, b, z)_{_{-1}} &=& (r_{\mathcal{B}}(b)x)\cdot z+l_{\mathcal{B}}(l_{\mathcal{A}}(x)b)z+ r_{\mathcal{A}}(z)(l_{\mathcal{A}}(x)b)+ x \cdot (l_{\mathcal{B}}(b)z)+r_{\mathcal{B}}(r_{\mathcal{A}}(z)b)x \cr &+& l_{\mathcal{A}}(x)(r_{\mathcal{A}}(z)b),\cr
(a, y, z)_{_{-1}} &=& (l_{\mathcal{B}}(a)y)\cdot z + l_{\mathcal{B}}(r_{\mathcal{A}}(y)a)z+ r_{\mathcal{A}}(z)(r_{\mathcal{A}}(y)a)
+l_{\mathcal{B}}(a)(y \cdot z)
+r_{\mathcal{A}}(y \cdot z )a,\cr
(a, y, c)_{_{-1}} &=& r_{\mathcal{B}}(c)(l_{\mathcal{B}}(a)y)+ (r_{\mathcal{A}}(y)a) \circ c +l_{\mathcal{A}}(l_{\mathcal{B}}(a)y)c+ l_{\mathcal{B}}(a)(r_{\mathcal{B}}(c)y)+a \circ (l_{\mathcal{A}}(y)c) \cr &+& r_{\mathcal{A}}(r_{\mathcal{B}}(c)y)a, \cr
(a, b, z)_{_{-1}} &=& l_{\mathcal{B}}(a \circ b)z +r_{\mathcal{A}}(z)(a \circ b)+l_{\mathcal{B}}(a)(l_{\mathcal{B}}(b)z)+r_{\mathcal{A}}(l_{\mathcal{B}}(b)z)a +r_{\mathcal{A}}(l_{\mathcal{B}}(b)z)a \cr
&+& a \circ (r_{\mathcal{A}}(z)b),
\eeqs
which can also be re-expressed as:
\beq
(x+a, y+b, z+c)_{_{-1}}&= &(x, y, z)_{_{-1}}+(x, y, c)_{_{-1}}+(x, b, z)_{_{-1}}+ (x, b, c)_{_{-1}} \cr
&+& 
 (a, y, z)_{_{-1}}+(a, y, c)_{_{-1}}+(a, b, z)_{_{-1}} +(a, b, c)_{_{-1}}.
\eeq
Similarly, 
\beq
(y+b, x+a, z+c)_{_{-1}} &=& (y, x, z)_{_{-1}}+(b, x, z)_{_{-1}}+(y, x, c_{_{-1}})+ (b, a, z)_{_{-1}}\cr &+&(y, a, z)_{_{-1}}+(y, a, c)_{_{-1}}+ (b, a, c)_{_{-1}} +(b, x, c)_{_{-1}}.
\eeq
Since $\A$ and $\B$ are pre-JJ algebras, we have
\beqs
(x,y,z)_{_{-1}} &=& -(y,x,z)_{_{-1}} \\
(a,b,c)_{_{-1}} &=& -(b,a,c)_{_{-1}}.
\eeqs
Hence, 
\beqs
(x,b,z)_{_{-1}} = -(b,x,z)_{_{-1}} &\Leftrightarrow& (y,a,z)_{_{-1}} = -(a,y,z)_{_{-1}} \{x\rightarrow y,b\rightarrow a,z\rightarrow z \} \\
(x,b,c)_{_{-1}}=-(b,x,c)_{_{-1}} &\Leftrightarrow& (a,y,c)_{_{-1}} =-(y,a,c)_{_{-1}} \{x\rightarrow y,b\rightarrow a,c\rightarrow c\}.
\eeqs
Then, it remains to show that:
\beq\label{eq6}
(x,a,y)_{_{-1}} &=& -(a,x,y)_{_{-1}},
\eeq
\beq \label{eq7}
(x,a,b)_{_{-1}} &=& -(a,x,b)_{_{-1}},
\eeq
\beq \label{eq8}
(x,y,a)_{_{-1}} &=& -(y,x,a)_{_{-1}},
\eeq
\beq\label{eq9}
(a,b,x)_{_{-1}} &=& -(b,a,x)_{_{-1}}.
\eeq
We have
\beqs
\eqref{eq6} &\Leftrightarrow& l_{\mathcal{B}}(l_{\mathcal{A}}(x)a)y+r_{\mathcal{A}}(y)(l_{\mathcal{A}}(x)a)+(r_{\mathcal{B}}(a)x)y+x(l_{\mathcal{B}}(a)y)+
l_{\mathcal{A}}(x)(r_{\mathcal{A}}(y)a) \cr 
&&+r_{\mathcal{B}}(r_{\mathcal{A}}(y)a)x =-\{(l_{\mathcal{B}}(a)x)y+l_{\mathcal{B}}(r_{\mathcal{A}}(x)a)y+r_{\mathcal{A}}(y)(r_{\mathcal{A}}(x)a) \cr
&&+l_{\mathcal{B}}(a)(xy)+r_{\mathcal{A}}(xy)a\}\cr
&\Leftrightarrow& -l_{\mathcal{B}}(a)(xy)=
l_{\mathcal{B}}(l_{\mathcal{A}}(x)a)y+(r_{\mathcal{B}}(a)x)y+x(l_{\mathcal{B}}(a)y)+
 \cr 
&&+r_{\mathcal{B}}(r_{\mathcal{A}}(y)a)x +(l_{\mathcal{B}}(a)x)y+l_{\mathcal{B}}(r_{\mathcal{A}}(x)a)y \cr
&&+l_{\mathcal{B}}(a)(xy)\cr
\eqref{eq6}&\Leftrightarrow&\eqref{eqq4}
\eeqs
since $[l_x, r_y]=-r_{x\cdot y}-r_yr_x$.
\beqs
\eqref{eq7} &\Leftrightarrow& r_{\mathcal{B}}(b)(r_{\mathcal{B}}(a)x)+(l_{\mathcal{A}}(x)a)\circ b + 
l_{\mathcal{A}}(r_{\mathcal{B}}(a)x)b+r_{\mathcal{B}}(a \circ b)x \cr
&&+l_{\mathcal{A}}(x)(a \circ b) = -\{l_{\mathcal{A}}(l_{\mathcal{B}}(a)x)b+r_{\mathcal{B}}(b)(l_{\mathcal{B}}(a)x)+
(r_{\mathcal{A}}(x)a) \circ b\\
&&+a \circ (l_{\mathcal{A}}(x)b)+l_{\mathcal{B}}(a)(r_{\mathcal{B}}(b)x)+r_{\mathcal{A}}(r_{\mathcal{B}}(b)x)a\} \\
&\Leftrightarrow& -l_{\mathcal{A}}(x)(a \circ b) = l_{\mathcal{A}}(l_{\mathcal{B}}(a)x+r_{\mathcal{B}}(a)x)b +
(l_{\mathcal{A}}(x)a+r_{\mathcal{A}}(x)a)\circ b \cr
&&+r_{\mathcal{A}}(r_{\mathcal{B}}(b)x)a+a\circ (l_{\mathcal{A}}(x)b) \mbox{ with, } (l_{\mathcal{B}}, r_{\mathcal{B}}) 
\mbox{ is bimodule of } \mathcal{B} \cr
\eqref{eq7} &\Leftrightarrow& \eqref{eqq2} \mbox{ with, } (l_{\mathcal{B}}, r_{\mathcal{B}}) \mbox{ is bimodule of } \mathcal{B}.
\eeqs
\beqs
\eqref{eq8} &\Leftrightarrow& l_{\mathcal{A}}(xy)a+r_{\mathcal{B}}(a)(xy)+x(r_{\mathcal{B}}(a)y)+l_{\mathcal{A}}(x)(l_{\mathcal{A}}(y)a)+
r_{\mathcal{B}}(l_{\mathcal{A}}(y)a)x \cr
&& = -\{l_{\mathcal{A}}(yx)a+r_{\mathcal{B}}(a)(yx)+y(r_{\mathcal{B}}(a)x)+l_{\mathcal{A}}(y)(l_{\mathcal{A}}(x)a)+
r_{\mathcal{B}}(l_{\mathcal{A}}(x)a)y\} \cr
&\Leftrightarrow& r_{\mathcal{A}}(a)[x,y]=r_{\mathcal{B}}(l_{\mathcal{A}}(y)a)x+r_{\mathcal{B}}(l_{\mathcal{A}}(x)a)y+
x(r_{\mathcal{B}}(a)y)+y(r_{\mathcal{B}}(a)x) \cr
\eqref{eq8} &\Leftrightarrow& \eqref{eqq3},
\eeqs
since \ref{eqbimodule2} hold.
\beqs
\eqref{eq9} &\Leftrightarrow& l_{\mathcal{B}}(a \circ b)x+r_{\mathcal{A}}(x)(a \circ b)+a \circ (r_{\mathcal{A}}(x)b)+
l_{\mathcal{B}}(a)(l_{\mathcal{B}}(b)x) +r_{\mathcal{A}}(l_{\mathcal{B}}(b)x)a  \cr
&&= -\{l_{\mathcal{B}}(b \circ a)x+r_{\mathcal{A}}(x)(b \circ a)+b \circ (r_{\mathcal{A}}(x)a)+l_{\mathcal{B}}(b)(l_{\mathcal{B}}(a)x)+
r_{\mathcal{A}}(l_{\mathcal{B}}(a)x)b\}\cr
\eqref{eq9} &\Leftrightarrow& r_{\mathcal{A}}(x)([a,b]) = r_{\mathcal{A}}(l_{\mathcal{B}}(b)x)a+r_{\mathcal{A}}(l_{\mathcal{B}}(a)x)b+
a \circ (r_{\mathcal{A}}(x)b)+b \circ (r_{\mathcal{A}}(x)a) \mbox{ since \ref{eqbimodule2} hold, } \cr
\eqref{eq9} &\Leftrightarrow &\eqref{eqq1} \mbox{ and }l_{\mathcal{B}} 
\mbox{ is a linear representation of the sub-adjacent}\cr 
&& \mbox{Jacobi-Jordan algebra } \mathcal{G}(\mathcal{B}).
\eeqs
Hence, $\mathcal{A}\bowtie^{-1} \mathcal{B}$ is an pre-JJ algebra if and only if  $(l_{\mathcal{A}}, r_{\mathcal{A}})$ is a  bimodule of 
$\mathcal{\mathcal{A}}$ and $(l_{\mathcal{B}}, r_{\mathcal{B}})$ is a bimodule of $\mathcal{B}$ and
equations \eqref{eqq1} - \eqref{eqq4} hold.

On the other hand, if $\mathcal{A}\mbox{ and } \mathcal{B}$ are pre-JJ sub-algebras of a pre-JJ algebra $C$ such 
that $C=\mathcal{A}\oplus \mathcal{B}$ which is a direct sum of the underlying
vector spaces of $\mathcal{A}$ and $\mathcal{B}$, then the linear maps 
\beqs
l_{\mathcal{A}},r_{\mathcal{A}} : \mathcal{A} \rightarrow \mathfrak{gl}(\mathcal{B}),  \ \ \ 
l_{\mathcal{B}},r_{\mathcal{B}}: \mathcal{B} \rightarrow \mathfrak{gl}(\mathcal{A}).
\eeqs
defined by 
\beq
x \ast a = l_{\mathcal{A}}(x)a+r_{\mathcal{B}}(a)x \\
a \ast x = l_{\mathcal{B}}(a)x+r_{\mathcal{A}}(x)a 
\eeq
satisfy the equations \eqref{eqq1}  - \eqref{eqq4}. In addition, $(l_{\mathcal{A}},r_{\mathcal{A}})$ is a bimodule of $\mathcal{A}$ and 
$(l_{\mathcal{B}},r_{\mathcal{B}})$ is a bimodule of $\mathcal{B}$. $\hfill{\square}$\\
\begin{corollary}\label{corr}
Let $(\A, \B, l_{\A}, r_{\A}, l_{\B}, r_{\B})$ be 
a matched pair of pre-JJ algebras. Then, 
$(~\mathcal{G}(\A), \mathcal{G}(\B), l_{ \A}+r_{ \A} , l_{\B}+r_{\B} ~)$ is a matched pair of sub-adjacent JJ algebras $\G(\A)$ and~$\G(\B).$
\end{corollary}
\textbf{Proof:} 

 By using the Proposition~\ref{propcentlie} and the bimodules
$\displaystyle (l_{\A}, r_{\A}, \B)$ and  $(l_{\B}, r_{\B}, \A),$ 
 we have: 
$\ad_{\A}:=l_{\A}+r_{\A}$ and $\ad_{\B}:=l_{\B}+r_{\B}$ are   the linear representations of the sub-adjacent JJ algebras $\G(\A)$ and $\G(\B)$ of the pre-JJ algebras $\A$ and $\B,$ respectively.
Then, the statement that
$\displaystyle\mathcal{G}{(\A)}\bowtie_{\ad_{\B}}^{-1,\ad_{\A}}\mathcal{G}{(\B)}$ is a matched pair of the JJ algebras $\G(\A)$ and $\G(\B)$ follows from Theorem~\ref{theoo}. By analogous step giving: 
\beq\label{ad1}
\ad_{\A}(x)\left[a, b\right]-\left[\ad_{\A}(x)a, b\right]-\left[a, \ad_{\A}(x)b\right]\cr 
-\ad_{\A}(\ad_{\B}(a)x)b-\ad_{\A}(\ad_{\B}(b)x)a=0
\eeq
\beq\label{ad2}
\ad_{\B}(a)\left[x, y\right]-\left[\ad_{\B}(a)x, y\right]-\left[x, \ad_{\B}(a)y\right]\cr 
-\ad_{\B}(\ad_{\A}(x)a)y-\ad_{\B}(\ad_{\A}(y)a)x=0  
\eeq
Thus, we first have:
\beq
&&\ad_{\A}(x)\left[a, b\right]-\left[\ad_{\A}(x)a, b\right]-\left[a, \ad_{\A}(x)b\right]-\ad_{\A}(\ad_{\B}(a)x)b-\ad_{\A}(\ad_{\B}(b)x)a\cr
&=& 
(l_{\A}+r_{\A})(x)\left[a, b\right]-\left[(l_{\A}+r_{\A})(x)a, b\right]-\left[a, (l_{\A}+r_{\A})(x)b\right] \cr 
&-& (l_{\A}+r_{\A})((l_{\B}+r_{\B})(a)x)b- 
(l_{\A}+r_{\A})((l_{\B}+r_{\B})(b)x)a  \cr
&=& 
l_{\A}(x)[a, b]+r_{\A}(x)[a, b]-
\left[l_{\A}(x)a, b\right]-
\left[r_{\A}(x)a, b\right]-
\left[a, l_{\A}(x)b\right]\cr
&-& \left[a, r_{\A}(x)b\right]- 
l_{\A}(l_{\B}(a)x)b-
l_{\A}(r_{\B}(a)x)b-
r_{\A}(l_{\B}(a)x)b-
r_{\A}(r_{\B}(a)x)b\cr
&-&  l_{\A}(l_{\B}(b)x)a- 
l_{\A}(r_{\B}(b)x)a-
r_{\A}(l_{\B}(b)x)a-
r_{\A}(r_{\B}(b)x)a \cr
&=& 
l_{\A}(x)(a \circ b)+
l_{\A}(x)(b \circ a)+
r_{\A}(x)(a \circ b)+
r_{\A}(x)(b \circ a)-
(l_{\A}(x)a)\circ b \cr
&-& b \circ(l_{\A}(x)a)-
(r_{\A}(x)a)\circ b-
b\circ (r_{\A}(x)a)-
a \circ (l_{\A}(x)b)-
(l_{\A}(x)b)\circ a \cr
&-& a \circ (r_{\A}(x)b)-
(r_{\A}(x)b)\circ a-
l_{\A}(l_{\B}(a)x)b-
l_{\A}(r_{\B}(a)x)b-
r_{\A}(l_{\B}(a)x)b \cr
&-& r_{\A}(r_{\B}(a)x)b -
l_{\A}(l_{\B}(b)x)a- 
l_{\A}(r_{\B}(b)x)a-
r_{\A}(l_{\B}(b)x)a-
r_{\A}(r_{\B}(b)x)a \cr
&=& 
\{ r_{\mathcal{A}}(x)([a,b])  -r_{\mathcal{A}}(l_{\mathcal{B}}(b)x)a-r_{\mathcal{A}}(l_{\mathcal{B}}(a)x)b-
a \circ (r_{\mathcal{A}}(x)b)-b \circ (r_{\mathcal{A}}(x)a)\}\cr
&+&
\{-l_{\mathcal{A}}(x)(a \circ b) - l_{\mathcal{A}}(l_{\mathcal{B}}(a)x-r_{\mathcal{B}}(a)x)b -
(l_{\mathcal{A}}(x)a-r_{\mathcal{A}}(x)a)\circ b \cr
&-&r_{\mathcal{A}}(r_{\mathcal{B}}(b)x)a-a\circ (l_{\mathcal{A}}(x)b)\}\cr
&+&
\{-l_{\mathcal{A}}(x)(b \circ a) - l_{\mathcal{A}}(l_{\mathcal{B}}(b)x-r_{\mathcal{B}}(b)x)a -
(l_{\mathcal{A}}(x)b-r_{\mathcal{A}}(x)b)\circ a \cr
&-&r_{\mathcal{A}}(r_{\mathcal{B}}(a)x)b-b\circ (l_{\mathcal{A}}(x)a)\}= 0. \nonumber
\eeq
Secondly:
\beq
&& \ad_{\B}(a)\left[x, y\right]-\left[\ad_{\B}(a)x, y\right]-\left[x, \ad_{\B}(a)y\right]-\ad_{\B}(\ad_{\A}(x)a)y-\ad_{\B}(\ad_{\A}(y)a)x \cr 
&=& 
 (l_{\B}+r_{\B})(a)\left[x, y\right]-\left[(l_{\B}+r_{\B})(a)x, y\right]-
 \left[x, (l_{\B}+r_{\B})(a)y\right] \cr 
 &-&(l_{\B}+r_{\B})((l_{\A}+r_{\A})(x)a)y- 
 (l_{\B}+r_{\B})((l_{\A}+r_{\A})(y)a)x 
 \cr
 &=&
 l_{\B}(a)[x, y]+r_{\B}(a)[x, y]-
 \left[l_{\B}(a)x, y\right]-
 \left[r_{\B}(a)x, y\right]-
 \left[x, l_{\B}(a)y\right]\cr
 &-& \left[x, r_{\B}(a)y\right]- 
 l_{\B}(l_{\A}(x)a)y-
 l_{\B}(r_{\A}(x)a)y-
 r_{\B}(l_{\A}(x)a)y-
 r_{\B}(r_{\A}(x)a)y\cr
&-&   l_{\B}(l_{\A}(y)a)x-
 l_{\B}(r_{\A}(y)a)x-
 r_{\B}(l_{\A}(y)a)x-
 r_{\B}(r_{\A}(y)a)x \cr 
 &=& 
 \{r_{\mathcal{A}}(a)[x,y]-r_{\mathcal{B}}(l_{\mathcal{A}}(y)a)x-r_{\mathcal{B}}(l_{\mathcal{A}}(x)a)y-
x(r_{\mathcal{B}}(a)y)-y(r_{\mathcal{B}}(a)x)\}\cr
&+&\{ -l_{\mathcal{B}}(a)(xy)-
l_{\mathcal{B}}(l_{\mathcal{A}}(x)a)y-(r_{\mathcal{B}}(a)x)y-x(l_{\mathcal{B}}(a)y)-r_{\mathcal{B}}(r_{\mathcal{A}}(y)a)x \cr
&-&(l_{\mathcal{B}}(a)x)y-l_{\mathcal{B}}(r_{\mathcal{A}}(x)a)y -l_{\mathcal{B}}(a)(xy)\}+\{-l_{\mathcal{B}}(a)(yx)-
l_{\mathcal{B}}(l_{\mathcal{A}}(y)a)x\cr
&-&(r_{\mathcal{B}}(a)y)x-y(l_{\mathcal{B}}(a)x)-r_{\mathcal{B}}(r_{\mathcal{A}}(x)a)y 
-(l_{\mathcal{B}}(a)y)x-l_{\mathcal{B}}(r_{\mathcal{A}}(y)a)x -l_{\mathcal{B}}(a)(yx)\}= 0.\nonumber
\eeq 
 Thus, we obtain
\beqs
\ad_{\A}(x)\left[a, b\right]-\left[\ad_{\A}(x)a, b\right]
-\left[a, \ad_{\A}(x)b\right]\cr 
-\ad_{\A}(\ad_{\B}(a)x)b-\ad_{\A}(\ad_{\B}(b)x)a=0,  
\eeqs
\beqs
\ad_{\B}(a)\left[x, y\right]-\left[\ad_{\B}(a)x, y\right]
-\left[x, \ad_{\B}(a)y\right]\cr 
-\ad_{\B}(\ad_{\A}(x)a)y-\ad_{\B}(\ad_{\A}(y)a)x=0.  
\eeqs
Hence, $(~\mathcal{G}(\A), \mathcal{G}(\B), \ad_{\A} , \ad_{\B})$ is a matched pair of sub-adjacent JJ algebras $\G(\A)$ and $\G(\B)$.$\cqfd$

\begin{definition}\label{dfnn}
Let $(l, r, V)$ be a bimodule of a pre-JJ algebra $\A,$ where $V$ is a finite dimensional vector 
space. The dual maps $l^{*}, r^{*}$  of the linear maps $l, r,$ are defined, respectively, as: 
$\displaystyle l^{*}, r^{*}: \A \rightarrow \mathfrak{gl}(V^{*})$
 such that:for all $x \in \A, u^{*} \in V^{*}, v \in V,$
\beq\label{dual1}
 l^*: \A & \longrightarrow & \mathfrak{gl}(V^*)  \cr
  x  & \longmapsto & l^*_x:
      \begin{array}{llll}
 V^* &\longrightarrow & V^* \\ 
  u^* & \longmapsto & l^*_x u^*: 
      \begin{array}{llll}
V  &\longrightarrow&  \K \cr
v  &\longmapsto& \left< l^{*}_xu^{*}, v \right>
 := \left<  u^{*}, l_x v\right>, 
      \end{array}
     \end{array}
 \eeq
\beq\label{dual2}
 r^*: \A & \longrightarrow & \mathfrak{gl}(V^*)  \cr
  x  & \longmapsto & r^*_x:
      \begin{array}{llll}
 V^* &\longrightarrow & V^* \\ 
  u^* & \longmapsto & r^*_x u^*: 
      \begin{array}{lllll}
V  &\longrightarrow&  \K \cr
v  &\longmapsto& \left< r^{*}_xu^{*}, v \right>
 := \left<  u^{*}, r_x v\right>. 
      \end{array}
     \end{array}
\eeq
\end{definition}  
\begin{proposition}
Let  $(\A, \cdot)$ be a pre-JJ algebra and   
$\displaystyle l, r: \A \rightarrow \mathfrak{gl}(V)$ be two linear maps, where $V$ is a finite dimensional vector space.
The following conditions are equivalent:
\begin{enumerate}
\item $(l, r, V)$ is a bimodule of $\A$.
\item $(r^{*}, l^{*}, V^{*})$ is a bimodule of
 $\A .$  
\end{enumerate}
\end{proposition}
{\textbf{Proof:}} 
\begin{enumerate}
\item [(1)$\displaystyle \Rightarrow$(2)]
Suppose that $\displaystyle (l, r, V) $ is a bimodule of $(\displaystyle \A, \cdot)$ and  show that 
$\displaystyle (r^*, l^*, V^*)$ is also a bimodule of $(\displaystyle \A, \cdot).$ We have:
\begin{itemize}
\item
\beqs
 \left<(l^*_{xy}+l^*_xl^*_y)u^*, v\right>
 &=& \left<l^*_{xy}u^*, v\right>+\left<(l^*_xl^*_y)u^*, v\right> 
=\left<l_{xy}(v), u^*\right>+\left<l_y(l_x(v)), u^*\right> \\
 &=&
  \left<(l_{xy}+l_yl_x)(v), u^*\right> \ 
 = \left<(-l_{yx}-l_xl_y)(v), u^*\right> \\
 &=& 
 -\left<l_{yx}(v), u^*\right>-\left<(l_xl_y)(v), u^*\right> \\
 &=&-\left<l_{yx}^*u^*, v\right>-\left<(l_y^*l_x^*)u^*, v\right> 
 =-\left<(l_{yx}^*+l_y^*l_x^*)u^*, v\right>.
\eeqs
Therefore,  
\beq\label{ici}
 l_{xy}^*+l_x^*l_y^*= 
 -l_{yx}^*-l_y^*l_x^*, \; \forall \;  x, y \;  \in\A
\eeq
\item 
\beqs
\left<[l^*_x, r_y^*]u^*, v\right> 
&=& \left<l_x^*(r_y^*)u^*, v\right>+
\left<r_y^*(l_x^*)u^*,v\right> 
=\left<l_x(v), r_y^*u^*\right>+
\left<r_yv, l_x^*u^*\right> 
\\
&=&
 \left<r_y(l_x(v)), u^*\right>+
\left<l_x(r_y(v)), u^*\right> 
 =\left<[r_y, l_x]v,u^*\right>  \\
&=& 
\left<-[r_x, l_y]v, u^*\right> = -\left<(r_x(l_y)-l_y(r_x))v, u^*\right> \\
&=&
 \left<-(l^*_yr^*_x+r_x^*l_y^*)u^*, v\right> 
=\left<-[l^*_y, r^*_x]u^*, v\right>
\eeqs
Therefore 
\beq\label{icii}
[l_x^*, r_y^*]=-[l_y^*, r_x^*], \; \forall \; x, y \; \in \A.
\eeq
\end{itemize}
By considering the relations \eqref{ici} and \eqref{icii}, we conclude that $\displaystyle (r^*, l^*, V)$ is a bimodule of $(\A, \cdot).$
\item [(2)$\displaystyle \Rightarrow$(1)] The converse, (i.e., by supposing that  $\displaystyle (r^*, l^*, V)$ is a bimodule of $(\A,\cdot)$ then  $\displaystyle (l, r, V)$ is also a bimodule of $(\A, \cdot)$), can be proved by direct calculations
by using similar relations as for the first part of the proof.$\cqfd$
\end{enumerate}
\begin{theorem}\label{ttheo}
Let $(\A, \cdot )$ be a pre-JJ algebra. Suppose that there exists a pre-JJ algebra structure $"\circ "$ on its dual space $\A^{*}.$ Then,  $(\A, \A^{*}, R_{\cdot}^*, L_{\cdot }^*, R_{\circ}^*, L_{\circ}^*)$ is a matched pair of pre-JJ algebras $\A$ and $\A^{*}$
if and only if \\$(\G(\A), \G(\A^{*}), -\ad^*_{\cdot}, -\ad^{*}_{\circ})$ is a matched pair of JJ algebras $\G(\A)$ and $\G(\A^{*}).$
\end{theorem}
{\textbf{Proof:}}

By considering the Theorem~\ref{theoo}, setting   
$\displaystyle l_{\A}:=R^{*}_{\cdot}, r_{\A}:=L^*_{\cdot}, l_{\B}:=R^*_{\circ},  
r_{\B}:=L^*_{\circ},$ and exploiting the
 Definition~{\ref{dfnmatched}}
  with $\G:=\G(\A),
   \h:=\G(\A^*),\rho:=R_{\cdot}^*+L_{\cdot}^*, \mu:=R_{\circ}^*+L_{\circ}^*,$ and  the relations
\eqref{ad1} and \eqref{ad2}, we have
\begin{itemize}
\item
The equation~$\eqref{eqt1}$ is equivalent to both the equations~$\eqref{eqq1}$ and 
~$\eqref{eqq3},$ i.e.:  
\beqs
&& (R^*_{\cdot}+L^*_{\cdot})(x)[a, b]-
\left[(R^*_{\cdot}+L^*_{\cdot})(x)a, b\right]- 
\left[a,(R^*_{\cdot}+L^*_{\cdot})(x)b\right] \\
&-& (R^*_{\cdot}+L^*_{\cdot})((R^*_{\circ}+L^*_{\circ})(a)x)b- 
(R^*_{\cdot}+L^*_{\cdot})((R^*_{\circ}+L^*_{\circ}(b)x)a\\
&=& R_{\cdot}^*(x)[a, b]+L_{\cdot}^*(x)[a, b]-[R_{\cdot}^*(x)a, b]-[L_{\cdot}^*(x)a,b]-[a, R_{\cdot}^*(x)b]-[a, L_{\cdot}^*(x)b]-R_{\cdot}^*(R_{\circ}^*(a)x)b\cr
&-&R_{\cdot}^*(L_{\circ}^*(a)x)b-R_{\cdot}(R^*_{\circ}(b)x)a-R_{\cdot}^*(L_{\circ}^*(b)x)a-L_{\cdot}^*(L_{\circ}^*(a)x)b-L_{\cdot}^*(R_{\circ}(a)x)b \\
&-& L_{\cdot}^*(L_{\circ}^*(b)x)a-L_{\cdot}^*(R_{\circ}^*(b)x)a \\
&=&R_{\cdot}^*(a\circ b)+R_{\cdot}^*(b\circ a)+L_{\cdot}^*(x)[a, b]-(R_{\cdot}^*(x)a)\circ b-b\circ(R_{\cdot}^*(x)a)-(L_{\cdot}^*(x)a)\circ b-b\circ(L_{\cdot}^*(x)a)\cr
&-&a\circ(R_{\cdot}^*(x)b)-(R_{\cdot}^*(x)b)\circ a-a\circ(L_{\cdot}^*(x)b)-(L_{\cdot}^*(x)b)\circ a-R_{\cdot}^*(R_{\circ}^*(a)x)b-R_{\cdot}^*(L_{\circ}^*(a)x)b\cr
&-&R_{\cdot}(R^*_{\circ}(b)x)a-R_{\cdot}^*(L_{\circ}^*(b)x)a-L_{\cdot}^*(L_{\circ}^*(a)x)b-L_{\cdot}^*(R_{\circ}(a)x)b-L_{\cdot}^*(L_{\circ}^*(b)x)a-L_{\cdot}^*(R_{\circ}^*(b)x)a \\
&=&\{R_{\cdot}^*(a\circ b)-R_{\cdot}^*(R_{\circ}^*(a)x)b-R_{\cdot}^*(L_{\circ}^*(a)x)b-(R_{\cdot}^*(x)a)\circ b-(L_{\cdot}^*(x)a)\circ b-L_{\cdot}^*(L_{\circ}^*(b)x)a\cr
&-&a\circ(R_{\cdot}^*(x)b)\}+ \{R_{\cdot}^*(b\circ a)-R_{\cdot}^*(R_{\circ}^*(b)x)a-R_{\cdot}^*(L_{\circ}^*(b)x)a-(R_{\cdot}^*(x)b)\circ a-(L_{\cdot}^*(x)b)\circ a\cr
&-&b\circ(R_{\cdot}^*(x)a)-L_{\cdot}^*(L_{\circ}^*(a)x)b\}+\{L_{\cdot}^*(x)[a, b]-b\circ(L_{\cdot}^*(x)a)-a\circ(L_{\cdot}^*(x)b)-L_{\cdot}^*(R_{\circ}(a)x)b\cr
&-&L_{\cdot}^*(R_{\circ}(b)x)a\}
=0.
\eeqs
The two first relations in brace on the last equality give zero (see \eqref{eqq2}) and  the last  one brace  also yields  zero (see (\ref{eqq1})). 
\item 
The equation~$\eqref{eqt2}$ is equivalent to both the equations~$\eqref{eqq3}$ and 
~$\eqref{eqq4},$ i.e.:
\beqs
&& (R^*_{\circ}+L^*_{\circ})(a)[x, y]-
\left[(R^*_{\circ}+L^*_{\circ})(a)x, y\right]- 
\left[x,(R^*_{\circ}+L^*_{\circ})(a)y\right] \\
&-&(R^*_{\circ}+L^*_{\circ})((R^*_{\cdot}+L^*_{\cdot})(x)a)y- 
(R^*_{\circ}+L^*_{\circ})((R^*_{\cdot}+L^*_{\cdot}(y)a)x\\
&=&R_{\cdot}^*(a)[x, y]+L_{\cdot}^*(a)[x, y]-[R_{\cdot}^*(a)x, y]-[L_{\cdot}^*(a)x,y]-[x, R_{\cdot}^*(a)y]-[x, L_{\cdot}^*(a)y]-R_{\cdot}^*(R_{\circ}^*(x)a)y\cr
&-&R_{\cdot}^*(L_{\circ}^*(x)a)y-R_{\cdot}(R^*_{\circ}(y)a)x-R_{\cdot}^*(L_{\circ}^*(y)a)x-L_{\cdot}^*(L_{\circ}^*(x)a)y-L_{\cdot}^*(R_{\circ}(x)a)y \\
&-& L_{\cdot}^*(L_{\circ}^*(y)a)x-L_{\cdot}^*(R_{\circ}^*(y)a)x \\
&=&R_{\cdot}^*(x\circ y)+R_{\cdot}^*(y\circ x)+L_{\cdot}^*(a)[x, y]-(R_{\cdot}^*(a)x)\circ y-y\circ(R_{\cdot}^*(a)x)-(L_{\cdot}^*(a)x)\circ y-y\circ(L_{\cdot}^*(a)x)\cr
&-&x\circ(R_{\cdot}^*(a)y)-(R_{\cdot}^*(a)y)\circ x-x\circ(L_{\cdot}^*(a)y)-(L_{\cdot}^*(a)y)\circ x-R_{\cdot}^*(R_{\circ}^*(x)a)y-R_{\cdot}^*(L_{\circ}^*(x)a)y\cr
&-&R_{\cdot}(R^*_{\circ}(y)a)x-R_{\cdot}^*(L_{\circ}^*(y)a)x-L_{\cdot}^*(L_{\circ}^*(x)a)y-L_{\cdot}^*(R_{\circ}(x)a)y-L_{\cdot}^*(L_{\circ}^*(y)a)x-L_{\cdot}^*(R_{\circ}^*(y)a)x \\
&=&\{R_{\cdot}^*(x\circ y)-R_{\cdot}^*(R_{\circ}^*(x)a)y-R_{\cdot}^*(L_{\circ}^*(x)a)y-(R_{\cdot}^*(a)x)\circ y-(L_{\cdot}^*(a)x)\circ y-L_{\cdot}^*(L_{\circ}^*(y)a)x\cr
&-&x\circ(R_{\cdot}^*(a)y)\}+ \{R_{\cdot}^*(y\circ x)-R_{\cdot}^*(R_{\circ}^*(y)a)x-R_{\cdot}^*(L_{\circ}^*(y)a)x-(R_{\cdot}^*(a)y)\circ x-(L_{\cdot}^*(a)y)\circ x\cr
&-&y\circ(R_{\cdot}^*(a)x)-L_{\cdot}^*(L_{\circ}^*(x)a)y\}+\{L_{\cdot}^*(a)[x, y]-y\circ(L_{\cdot}^*(a)x)-x\circ(L_{\cdot}^*(a)y)-L_{\cdot}^*(R_{\circ}(x)a)y\cr
&-&L_{\cdot}^*(R_{\circ}(y)a)x\
=0.
\eeqs
The two first relations in brace on the last equality gives zero (see \eqref{eqq3}) and the last one brace also  leads to zero (see \eqref{eqq4}). 

Therefore, hold the equivalences. $\cqfd$
\end{itemize}


\section{Double constructions of symmetric (pre)Jacobi-Jordan algebras}
   In this section, we define and establish the double constructions of the symmetric JJ
algebras and symmetric pre-JJ algebras. 
\begin{definition}
We call $(\mathcal{J}, B)$ a   double
construction of a symmetric JJ
algebra associated to
$\mathcal{J}_1$ and ${\mathcal J}_1^*$ if it satisfies the conditions
 \begin{enumerate}
 \item[(1)] $ \mathcal{J} = \mathcal{J}_{1}
\oplus \mathcal{J}^{\ast}_{1} $ as the direct sum of vector
spaces; \item[(2)] $\mathcal{J}_1$ and ${\mathcal J}_1^*$ are JJ subalgebras of $\mathcal{J}$; \item[(3)] $B$ is the natural non-degerenate invariant symmetric
bilinear form on $ \mathcal{J}_{1} \oplus \mathcal{J}^{\ast}_{1} $
given by
 \beq \label{quadratic form}
 B(x + a^{\ast}, y + b^{\ast}) = \langle x, b^{\ast} \rangle +  \langle a^{\ast}, y \rangle
\eeq
for all  $x, y \in \mathcal{J}_{1}, a^{\ast}, b^{\ast} \in \mathcal{J}^{\ast}_{1}$
where $ \langle  , \rangle $ is the natural pair between the vector space $ \mathcal{J}_{1} $ and its dual space  $ \mathcal{J}^{\ast}_{1} $.
\end{enumerate}
\end{definition} 

\begin{theorem}
Let $ (\mathcal{J}, \diamond) $ be a JJ algebra. Suppose that there is a JJ algebra algebra structure $ " \bullet " $ on its 
dual space $ \mathcal{J}^{\ast}$. Then, there is a double construction of a symmetric JJ algebra associated to $ (\mathcal{J}, \cdot) $ 
and $ (\mathcal{J}^{\ast}, \circ) $ if and only if $ (\mathcal{J}, \mathcal{J}^{\ast}, L^{\ast}_{\diamond}, L^{\ast}_{\bullet})$ 
is a matched pair of JJ algebras.
\end{theorem}

\begin{proof}  
From Theorem \ref{matchJacobyJordan}, we know that
$(\mathcal{J}\oplus\mathcal{J}^{\ast},\ast)$
is a JJ algebra, where $\ast$ is
given by
\begin{equation}\label{eqn:standbracket}
  (x+ a^{\ast})\ast(y+ b^{\ast})=x\diamond y + L^*_\diamond(x)b^{\ast}+ L^*_\diamond(y)a^{\ast} +  a^{\ast}\bullet b^{\ast} + L^*_\bullet(a^{\ast})y + L^*_\bullet(b^{\ast})x
\end{equation} if and only if $ (\mathcal{J}, \mathcal{J}^{\ast}, L^{\ast}_{\diamond}, L^{\ast}_{\bullet})$. Furthermore, B is invariant  with the product $\ast$, that is  $B[(x + a^{\ast})\ast 
(y + b^{\ast}),(z + c^{\ast})] = B[(x + a^{\ast}), (y + b^{\ast})\ast (z + c^{\ast})] $.  Indeed, we have
\beqs
B[(x + a^{\ast})\ast (y + b^{\ast}), (z + c^{\ast})]   
                                                               &=& \langle x\cdot y, c^{\ast} \rangle + 
                                                     \langle c^{\ast} \circ a^{\ast} , y \rangle + \langle b^{\ast} \circ c^{\ast} , x \rangle \cr
                                                             && + \langle a^{\ast} \circ b^{\ast} , z \rangle 
                                                             + \langle z\cdot x , b^{\ast} \rangle +
                                                              \langle y\cdot z , a^{\ast} \rangle\cr
                                      &=&   B[x + a^{\ast} ,(y + b^{\ast}) \ast (z + c^{\ast})]
\eeqs 
\end{proof}

\begin{definition}
We call $(\mathcal{A}, B)$ a   double
construction of a symmetric pre-JJ
algebra associated to
$\mathcal{A}_1$ and ${\mathcal A}_1^*$ if it satisfies the conditions
 \begin{enumerate}
 \item[(1)] $ \mathcal{A} = \mathcal{A}_{1}
\oplus \mathcal{A}^{\ast}_{1} $ as the direct sum of vector
spaces; \item[(2)] $\mathcal{A}_1$ and ${\mathcal A}_1^*$ are JJ subalgebras of $\mathcal{A}$; \item[(3)] $B$ is the natural non-degerenate invariant symmetric
bilinear form on $ \mathcal{A}_{1} \oplus \mathcal{A}^{\ast}_{1} $
given by
 \beq \label{quadratic form}
 B(x + a^{\ast}, y + b^{\ast}) = \langle x, b^{\ast} \rangle +  \langle a^{\ast}, y \rangle
\eeq
for all  $x, y \in \mathcal{A}_{1}, a^{\ast}, b^{\ast} \in \mathcal{A}^{\ast}_{1}$
where $ \langle  , \rangle $ is the natural pair between the vector space $ \mathcal{A}_{1} $ and its dual space  $ \mathcal{A}^{\ast}_{1} $.
\end{enumerate}
\end{definition} 
\begin{theorem} \label{thmdouble}
Let $ (\mathcal{A}, \cdot) $ be a pre-JJ algebra. Suppose that there is a pre-JJ algebra algebra structure $ " \circ " $ on its 
dual space $ \mathcal{A}^{\ast}$. Then, there is a double construction of a symmetric pre-JJ algebra associated to $ (\mathcal{A}, \cdot) $ 
and $ (\mathcal{A}^{\ast}, \circ) $ if and only if $ (\A, \A^*, R_{\cdot}^*, L_{\cdot}^*, R_{\circ}^*, L_{\circ}^*)$ 
is a matched pair of pre-JJ algebras.
\end{theorem}
{\textbf{Proof:}}

By considering that $(\A, \A^*, R_{\cdot}^*, L_{\cdot}^*; R_{\circ}^*, L_{\circ}^* )$ is a matched pair of pre-JJ algebras, 
it follows that the bilinear product $\ast$ defined in the Theorem~\ref{theoo} 
is pre-JJ  on the direct sum of underlying vectors spaces,  $\A \oplus \A^*.$ 
We have 
$\forall x, y, z \in \A; a, b, c \in \A^{*}$.
\beqs
\bb_{\A} ((x+a)\ast(y+b), z+c ) 
&=& \left< xy+R^*_{\circ}(a)y+L_{\circ}^*(b)x, c  \right>+
\left< z, a \circ b + R^*_{\cdot}(x)b+L^*_{\cdot}(y)a  \right> \cr
&=& \left<xy, c  \right>+ \left<R^*_{\circ}(a)y, c \right>+ 
\left< L^*_{\circ}(b)x , c\right> +
\left<z, a \circ b  \right>
+\left<z, R^*_{\cdot}(x)b \right>\cr&+&
\left<z, L^*_{\cdot}(y)a   \right> = 
\left<  xy, c \right>+
\left<y, R_a(c) \right>+
\left<x, L_b(c)  \right>+
\left< z, a \circ b  \right>\cr &+&
\left< R_x(z) , b \right>+
\left< L_y(z) , a \right> 
=
\left< xy, c\right>+
\left<  y, c\circ a \right>\cr &+&
\left< x , b \circ c  \right>+ 
\left< z, a \circ b  \right>+
\left<zx, b   \right> +
\left<yz, a \right>.
\eeqs

\beqs
\bb_{\A}\left((x+a), (y+b)\ast(z+c)\right)
&=& \left<x, b \circ c +R^*_{\cdot}(y)c+L^*_{\cdot}(z)b \right>+ 
<  yz+R^*_{\circ}(b)z\cr 
&+& L_{\circ}^*(c)y   ,a   > +
\left<x, b \circ c \right>+
\left<x, R^*_{\cdot}(y)c  \right>+
\left<x , L^*_{\cdot}(z)b  \right> \cr
&+&
\left< yz , a\right>+ 
\left<R^*_{\circ}(b)z, a   \right>+
\left<L^*_{\circ}(c)y, a  \right> \cr
&=&
\left<x, b\circ c   \right>+
\left<R_y(x),  c \right>+
\left<L_z(x), b  \right> \cr
&+&
\left< yz, a \right>+
\left<z, R_b(a) \right>+
\left<y, L_c(a) \right> \cr
&=&
\left<x, b\circ c  \right>+
\left<xy, c  \right>+
\left<zx, b  \right>+ 
\left<yz, a  \right>\cr &+&
\left<z, a \circ b \right>+
\left<y, c \circ a  \right>.
\eeqs
Therefore, the following relation 
\beq
\bb_{\A} ((x+a)\ast(y+b), (z+c) ) =\bb_{\A}\left((x+a), (y+b)\ast(z+c)\right)
\eeq 
holds,
which expresses the invariance of the standard bilinear form on $\A \oplus \A^*.$ 
Therefore, $(A\oplus \A^*, B)$ is the standard double construction of the pre-JJ algebras $\A$ and  $\A^*.$
$\cqfd$

\begin{proposition}\label{proposition_11}
Let $(\A, \cdot)$ be a pre-JJ algebra and $(\A^*, \circ)$ be a pre-JJ algebra structure on its dual space $\A^*.$ Then the following conditions are equivalent:
\begin{enumerate}
\item $(\A\oplus\A^*, B)$ is the standard double construction of considered pre-JJ algebras; 

\item $(\G(\A), \G(\A^*),-\ad_{\cdot}^*, -\ad_{\circ}^*)$ is a  matched pair of sub-adjacent JJ algebras;
\item $(\A, \A^*, R_{\cdot}^*, L_{\cdot}^*, R_{\circ}^*, L_{\circ}^*)$ is a matched pair of pre-JJ algebras.
\end{enumerate}
\end{proposition}
{\textbf{Proof:}}

From Theorem~\ref{ttheo},
$\displaystyle (2) \Longleftrightarrow (3),$ while from  Theorem~\ref{thmdouble} shows that 
$(1) \Longleftrightarrow (3).$ Then $(1) \Longleftrightarrow (2).$
$\cqfd$
%
\section{Computations in dimension two}

In this section, we investigate the classification of 2-dimensional complex pre-JJ algebras and some  double constructions. 

Let $\mathcal{A} $ be a pre-JJ algebra such that there is a pre-JJ structure $"\circ"$ on its dual space $\mathcal{A}^{\ast}$ spanned by $\{e_1,e_2\}$ and $\{e_1^{\ast},e_2^{\ast}\}$ respectively. Formula (\ref{CSA}) leads to the following relations:\\
\beq
(e_i\cdot e_j)\cdot e_k+e_i\cdot (e_j\cdot e_k)=-(e_j\cdot e_i)\cdot e_k -e_j\cdot (e_i\cdot e_k),\quad\quad where\quad i,j,k=1,2.
\eeq
Let $e_1\cdot e_1=a_1e_1+a_2e_2$, $e_1\cdot e_2=b_1e_1+b_2e_2$, $e_2\cdot e_1=c_1e_1+c_2e_2$, $e_2\cdot e_2=d_1e_1+d_2e_2$ where $a_1, a_2, b_1, b_2, c_1, c_2 d_1, d_2 \in \C$.

\begin{proposition}
There are four non-isomorphic 2-dimensional pre-JJ
algebras $\mathcal{A} $ given by the following:
\beq\label{class}
e_i\cdot e_j=0,\quad\quad e_1\cdot e_1=e_2,\quad\quad e_2\cdot e_1=e_2,\quad\quad e_2\cdot e_2=e_1.
\eeq
\end{proposition}
\begin{proof}
Let $\A$ be a 2-dimensional antiassociative
algebra  with basis $\{e_1,e_2\}$. Suppose $x,y,z\in \A $ such that $x=x_1e_1+x_2e_2$, $y=y_1e_1+y_2e_2$ and $z=z_1e_1+z_2e_2$ with $x_1,x_2,y_1,y_2,z_1,z_2\in \C$. By antiassociativity and ignoring the coefficients, we get the relations $$(e_i\cdot e_j)\cdot e_k=-e_i\cdot(e_j\cdot e_k),\quad i,j,k=1,2.$$
Setting $e_1e_1 =a_1e_1+a_2e_2$, $e_1e_2=b_1e_1+b_2e_2$, $e_2e_1=c_1e_1+c_2e_2$ and $e_2e_2=d_1e_1+d_2e_2$, in the previous relations one have  the following eight relations equivalent to a system of 32 equations:
\beq
\begin{cases}
~ 2a_1^2+a_2c_1+a_2b_1=0,\\ ~2a_1a_2+a_2c_2+a_2b_2=0,,\\
~a_1b_1+a_2d_1+b_1a_1+b_2b_1=0,\\ ~a_1b_2+a_2d_2+b_1a_2+b^2=0,\\
~b_1a_1+b_2c_1+c_1a_1+c_2b_1=-c_1a_1-c_2c_1-a_1c_1-a_2d_1,\\ ~b_1a_2+b_2c_2+c_1a_2+c_2b_2=-c_1a_2-c_2^2-a_1c_2-a_2d_2,\\
~b_1^2+b_2d_1+d_1a_1+d_2b_1=-c_1b_1-c_2d_1-b_1c_1-b_2d_1,\\
~b_1b_2+b_2d_2+d1a_2+d_2b_2=-c_1b_2-c_2d_2-b_1c_2-b2d_2,\\
~c_1a_1+c_2c_1+a_1c_1+a_2d_1=-a_1b_1-b_2c_1-c_1a_1-c_2b_1,\\
~c_1a_2+c_2^2+a_1c_2+a_2d_2=-b_1a_2-b_2c_2-c_1a_2-c_2b_2,\\
~c_1b_1+c_2d_1+b_1c_1+b_2d_1=-b_1^2-b_2d_1-d_1a_1-d_2b_1,\\
~c_1b_2+c_2d_2+b_1c_2+b_1c_2+b_2d_2=-b_1b_2-b_2d_2-d_1a_2-d_2b_2,\\
~d_1a_1+d_2c_1+c_1^2+c_2d_1=-d_1a_1-d_2c_1-c_1^2-c_2d_1,\\
~d_1a_2+d_2c_2+c_1c_2+c_2d2=-d_1a_2-d_2c_2-c_1c_2-c_2d_2,\\
~d_1b_1+d_2d_1+d_1c_1+d_2d_1=-d_1b_1-d_2d_1-d_1c_1-d_2d_1,\\
~d_1b_2+d_2^2+d_1c_2+d_2^2=-d_1b_2-d_2^2-d_1c_2-d_2^2.
\end{cases}
\eeq
For 
\begin{center}
$a_1=0$\\
$\quad\quad\quad\quad\quad\quad\quad\quad\quad\quad\quad\quad\quad\quad\quad\quad\quad\quad\quad\quad\quad a_2=0\Rightarrow b_2=0, c_2=0, d_2=0, c_1=0, b_1=0$\\
$\quad\quad\quad\quad \quad \quad class:\quad e_2e_2=d_1e_1$.
\end{center}
For 
\begin{center}
$a_1=0$\\
$\quad\quad\quad a_2\neq 0$\\
$\quad\quad\quad\quad\quad\quad\quad\quad\quad \quad \quad\quad\quad\quad\quad \quad \quad\quad\quad\quad\quad\quad b_1=0\Rightarrow c_1=0, d_1=0, d_2=0, c_2=0, b_2=0$\\
$\quad\quad\quad\quad \quad \quad class:\quad e_1e_1=a_2e_2$.
\end{center}
For
\begin{center}
$a_1=0$\\
$\quad\quad\quad\quad\quad\quad\quad\quad\quad\quad\quad\quad\quad\quad\quad\quad d_1\neq 0\Rightarrow c_1=0, d_2=0, b_1=0, b_2=0$\\
 $\quad\quad\quad\quad\quad\quad\quad\quad\quad\quad c_2\neq 0\Rightarrow a_2=0$\\
$\quad\quad\quad\quad \quad \quad class:\quad e_2e_1=c_2e_2$.
\end{center}
The last class is the trivial one ie. $e_ie_j=0$.
The other cases lead to an absurdity. Therefore, by isomorphism we get the following four classes: $e_ie_j=0$, $e_1e_1=e_2$, $e_2e_1=e_2$ and $e_2e_2=e_1$.
\end{proof}
Now, we discuss of the double constructions.\\
$Case(I).$ $e_1\cdot e_1=e_2$. 
The considered product on the dual space is $e_2^{\ast}\circ e_2^{\ast}= e_1^{\ast}.$

Using relation (\ref{pro}) when  $ l_{\mathcal{A}} = R^{\ast}, r_{\mathcal{A}} = L^{\ast}, l_{\mathcal{B}} = l_{\mathcal{A}^{\ast}} = R^{\ast}_{\circ}, r_{\mathcal{B}} = r_{\mathcal{A}^{\ast}} = L^{\ast}_{\circ} $, we obtain the double construction of pre-JJ algebra $ ( \mathcal{A}\oplus \mathcal{A}^{\ast}, \ast, \mathcal{B}) $ associated to $ (\mathcal{A}, \cdot) $ and $ (\mathcal{A}^{\ast}, \circ)$ given explicitly by the following relations:
\begin{align*}
(e_1+e_1^{\ast})\ast (e_1+e_1^{\ast})&=(e_1\cdot e_1 +R^{\ast}_{\circ}(e_1^{\ast})e_1+L^{\ast}_{\circ}(e_1^{\ast})e_1)+(e_1^{\ast}\circ e_1^{\ast}+R_{\cdot}^{\ast}(e_1)e_1^{\ast}+L_{\cdot}^{\ast}(e_1)e_1^{\ast} ),\\
&=e_2,\\
(e_1+e_1^{\ast})\ast (e_1+e_2^{\ast})&=(e_1\cdot e_1 +R^{\ast}_{\circ}(e_1^{\ast})e_1+L^{\ast}_{\circ}(e_2^{\ast})e_1)+(e_1^{\ast}\circ e_2^{\ast}+R_{\cdot}^{\ast}(e_1)e_2^{\ast}+L_{\cdot}^{\ast}(e_1)e_1^{\ast} ),\\
&=2e_2+ e_1^{\ast},\\
(e_1+e_1^{\ast})\ast (e_2+e_1^{\ast})&=(e_1\cdot e_2 +R^{\ast}_{\circ}(e_1^{\ast})e_2+L^{\ast}_{\circ}(e_1^{\ast})e_1)+(e_1^{\ast}\circ e_1^{\ast}+R_{\cdot}^{\ast}(e_1)e_1^{\ast}+L_{\cdot}^{\ast}(e_2)e_1^{\ast} ),\\
&=0,\\
(e_1+e_1^{\ast})\ast (e_2+e_2^{\ast})&=(e_1\cdot e_2 +R^{\ast}_{\circ}(e_1^{\ast})e_2+L^{\ast}_{\circ}(e_2^{\ast})e_1)+(e_1^{\ast}\circ e_2^{\ast}+R_{\cdot}^{\ast}(e_1)e_2^{\ast}+L_{\cdot}^{\ast}(e_2)e_1^{\ast} ),\\
&=e_2+ e_1^{\ast},\\
(e_1+e_2^{\ast})\ast (e_1+e_1^{\ast})&=(e_1\cdot e_1 +R^{\ast}_{\circ}(e_2^{\ast})e_1+L^{\ast}_{\circ}(e_1^{\ast})e_1)+(e_2^{\ast}\circ e_1^{\ast}+R_{\cdot}^{\ast}(e_1)e_1^{\ast}+L_{\cdot}^{\ast}(e_1)e_1^{\ast} ),\\
&=2e_2,\\
(e_1+e_2^{\ast})\ast (e_1+e_2^{\ast})&=(e_1\cdot e_1 +R^{\ast}_{\circ}(e_2^{\ast})e_1+L^{\ast}_{\circ}(e_2^{\ast})e_1)+(e_2^{\ast}\circ e_2^{\ast}+R_{\cdot}^{\ast}(e_1)e_2^{\ast}+L_{\cdot}^{\ast}(e_1)e_2^{\ast} ),\\
&=3e_2+3e_1^{\ast},\\
(e_2+e_1^{\ast})\ast (e_1+e_1^{\ast})&=(e_2\cdot e_1 +R^{\ast}_{\circ}(e_1^{\ast})e_1+L^{\ast}_{\circ}(e_2^{\ast})e_1)+(e_1^{\ast}\circ e_1^{\ast}+R_{\cdot}^{\ast}(e_2)e_1^{\ast}+L_{\cdot}^{\ast}(e_1)e_1^{\ast} ),\\
&=e_2,\\
(e_2+e_1^{\ast})\ast (e_1+e_2^{\ast})&=(e_2\cdot e_1 +R^{\ast}_{\circ}(e_1^{\ast})e_1+L^{\ast}_{\circ}(e_2^{\ast})e_2)+(e_1^{\ast}\circ e_2^{\ast}+R_{\cdot}^{\ast}(e_2)e_2^{\ast}+L_{\cdot}^{\ast}(e_1)e_1^{\ast} ),\\
&=0,\\
(e_2+e_2^{\ast})\ast (e_1+e_1^{\ast})&=(e_2\cdot e_1 +R^{\ast}_{\circ}(e_2^{\ast})e_1+L^{\ast}_{\circ}(e_1^{\ast})e_2)+(e_2^{\ast}\circ e_1^{\ast}+R_{\cdot}^{\ast}(e_2)e_1^{\ast}+L_{\cdot}^{\ast}(e_1)e_2^{\ast} ),\\
&=e_2+e_1^{\ast},\\
(e_2+e_2^{\ast})\ast (e_1+e_2^{\ast})&=(e_2\cdot e_1 +R^{\ast}_{\circ}(e_2^{\ast})e_1+L^{\ast}_{\circ}(e_2^{\ast})e_2)+(e_2^{\ast}\circ e_2^{\ast}+R_{\cdot}^{\ast}(e_2)e_2^{\ast}+L_{\cdot}^{\ast}(e_1)e_2^{\ast} ),\\
&=e_2+2e_1^{\ast},\\
(e_2+e_2^{\ast})\ast (e_2+e_1^{\ast})&=(e_2\cdot e_2 +R^{\ast}_{\circ}(e_2^{\ast})e_2+L^{\ast}_{\circ}(e_1^{\ast})e_2)+(e_2^{\ast}\circ e_1^{\ast}+R_{\cdot}^{\ast}(e_2)e_1^{\ast}+L_{\cdot}^{\ast}(e_2)e_2^{\ast} ),\\
&=0,\\
(e_2+e_2^{\ast})\ast (e_2+e_2^{\ast})&=(e_2\cdot e_2 +R^{\ast}_{\circ}(e_2^{\ast})e_2+L^{\ast}_{\circ}(e_2^{\ast})e_2)+(e_2^{\ast}\circ e_2^{\ast}+R_{\cdot}^{\ast}(e_2)e_2^{\ast}+L_{\cdot}^{\ast}(e_2)e_2^{\ast} ),\\
&=e_1^{\ast},\\
(e_1+e_2^{\ast})\ast (e_2+e_1^{\ast})&=(e_1\cdot e_2 +R^{\ast}_{\circ}(e_2^{\ast})e_2+L^{\ast}_{\circ}(e_1^{\ast})e_1)+(e_2^{\ast}\circ e_1^{\ast}+R_{\cdot}^{\ast}(e_1)e_1^{\ast}+L_{\cdot}^{\ast}(e_2)e_2^{\ast} ),\\
&=0,\\
(e_1+e_2^{\ast})\ast (e_2+e_2^{\ast})&=(e_1\cdot e_2 +R^{\ast}_{\circ}(e_2^{\ast})e_2+L^{\ast}_{\circ}(e_2^{\ast})e_1)+(e_2^{\ast}\circ e_2^{\ast}+R_{\cdot}^{\ast}(e_1)e_2^{\ast}+L_{\cdot}^{\ast}(e_2)e_2^{\ast} ),\\
&=e_2+2e_1^{\ast}),\\
(e_2+e_1^{\ast})\ast (e_2+e_1^{\ast})&=(e_2\cdot e_2 +R^{\ast}_{\circ}(e_1^{\ast})e_2+L^{\ast}_{\circ}(e_1^{\ast})e_2)+(e_1^{\ast}\circ e_1^{\ast}+R_{\cdot}^{\ast}(e_2)e_1^{\ast}+L_{\cdot}^{\ast}(e_2)e_1^{\ast} ),\\
&=0,\\
(e_2+e_1^{\ast})\ast (e_2+e_2^{\ast})&=(e_2\cdot e_2 +R^{\ast}_{\circ}(e_1^{\ast})e_2+L^{\ast}_{\circ}(e_2^{\ast})e_2)+(e_1^{\ast}\circ e_2^{\ast}+R_{\cdot}^{\ast}(e_2)e_2^{\ast}+L_{\cdot}^{\ast}(e_2)e_1^{\ast} ),\\
&=0.
\end{align*}
$Case(II).$ $e_2\cdot e_1=e_2$. 
The product on the dual space is given by:
$
e_i^{\ast}\circ e_j^{\ast}=0,\quad\quad i,j=1,2.
$
The double construction of pre-JJ algebra $ ( \mathcal{A}\oplus \mathcal{A}^{\ast}, \ast, \mathcal{B}) $ associated to $ (\mathcal{A}, \cdot) $ and $ (\mathcal{A}^{\ast}, \circ)$ given explicitly by the following relations:
\begin{align*}
(e_1+e_1^{\ast})\ast (e_1+e_1^{\ast})&=0 ,\\
(e_1+e_1^{\ast})\ast (e_1+e_2^{\ast})&=e_2^{\ast},\\
(e_1+e_1^{\ast})\ast (e_2+e_1^{\ast})&=0,\\
(e_1+e_1^{\ast})\ast (e_2+e_2^{\ast})&=e_2^{\ast},\\
(e_1+e_2^{\ast})\ast (e_1+e_1^{\ast})&=0,\\
(e_1+e_2^{\ast})\ast (e_1+e_2^{\ast})&=e_2^{\ast},\\
(e_2+e_1^{\ast})\ast (e_1+e_1^{\ast})&=e_2,\\
(e_2+e_1^{\ast})\ast (e_1+e_2^{\ast})&=e_2,\\
(e_2+e_2^{\ast})\ast (e_1+e_1^{\ast})&=e_2,\\
(e_2+e_2^{\ast})\ast (e_1+e_2^{\ast})&=e_1^{\ast},\\
(e_2+e_2^{\ast})\ast (e_2+e_1^{\ast})&=e_1^{\ast},\\
(e_2+e_2^{\ast})\ast (e_2+e_2^{\ast})&=e_1^{\ast},\\
(e_1+e_2^{\ast})\ast (e_2+e_1^{\ast})&=e_1^{\ast},\\
(e_1+e_2^{\ast})\ast (e_2+e_2^{\ast})&=e_1^{\ast} + e_2^{\ast},\\
(e_2+e_1^{\ast})\ast (e_2+e_1^{\ast})&=0,\\
(e_2+e_1^{\ast})\ast (e_2+e_2^{\ast})&=0.
\end{align*}
$Case(III).$ $e_2\cdot e_2=e_1$. 
The product on the dual space is given by the following relations:
$
e_2^{\ast}\circ e_1^{\ast}=e_2^{\ast}.
$
The double construction of pre-JJ algebra $ ( \mathcal{A}\oplus \mathcal{A}^{\ast}, \ast, \mathcal{B}) $ associated to $ (\mathcal{A}, \cdot) $ and $ (\mathcal{A}^{\ast}, \circ)$ given explicitly by the following relations:
\begin{align*}
(e_1+e_1^{\ast})\ast (e_1+e_1^{\ast})&=0,\\
(e_1+e_1^{\ast})\ast (e_1+e_2^{\ast})&=0,\\
(e_1+e_1^{\ast})\ast (e_2+e_1^{\ast})&=e_2+e_2^{\ast},\\
(e_1+e_1^{\ast})\ast (e_2+e_2^{\ast})&=e_2+e_2^{\ast},\\
(e_1+e_2^{\ast})\ast (e_1+e_1^{\ast})&=e_2^{\ast},\\
(e_1+e_2^{\ast})\ast (e_1+e_2^{\ast})&=0,\\
(e_2+e_1^{\ast})\ast (e_1+e_1^{\ast})&=e_2^{\ast},\\
(e_2+e_1^{\ast})\ast (e_1+e_2^{\ast})&=e_1+e_2^{\ast}),\\
(e_2+e_2^{\ast})\ast (e_1+e_1^{\ast})&=2e_2^{\ast},\\
(e_2+e_2^{\ast})\ast (e_1+e_2^{\ast})&=e_1,\\
(e_2+e_2^{\ast})\ast (e_2+e_1^{\ast})&=e_1+2e_2^{\ast},\\
(e_2+e_2^{\ast})\ast (e_2+e_2^{\ast})&=2e_1,\\
(e_1+e_2^{\ast})\ast (e_2+e_1^{\ast})&=e_2^{\ast},\\
(e_1+e_2^{\ast})\ast (e_2+e_2^{\ast})&=0,\\
(e_2+e_1^{\ast})\ast (e_2+e_1^{\ast})&=e_1+e_2+e_2^{\ast},\\
(e_2+e_1^{\ast})\ast (e_2+e_2^{\ast})&=2e_1+e_2+2e_2^{\ast}.
\end{align*}

\begin{remark}
Let $(\A,\cdot)$ be a one dimensional pre-JJ algebra with a basis $\{e_1\}$. We obtain the trivial class $e_ie_j=0$.
\end{remark}
\section{Concluding remarks}
In this work, we defined pre-JJ algebras and, discussed their bimodule and matched pair. We established the double construction of JJ algebras and pre-JJ algebras. Finally, we described some double contructions of symmetric pre-JJ algebras in dimension two.

\end{document}